\documentclass[10pt,a4paper,reqno]{amsart} \usepackage{epsfig,amssymb,latexsym,amscd,amsmath,amsthm,stmaryrd, bussproofs}
\usepackage[inline]{enumitem}
\usepackage{verbatim}
\usepackage[all,2cell]{xy} \xyoption{v2} \UseAllTwocells
\usepackage{draftwatermark} \SetWatermarkLightness{1} \SetWatermarkFontSize{5cm}
\theoremstyle{plain}
\newtheorem{teo}{Theorem}[section]
\newtheorem{theo}[teo]{Theorem}
\newtheorem{coro}[teo]{Corollary}
\newtheorem{lema}[teo]{Lemma}

\theoremstyle{remark}
\newtheorem{rema}[teo]{Remark}

\newtheorem{nota}[teo]{Notation}
\theoremstyle{definition}
\newtheorem{defi}[teo]{Definition}
\newtheorem{obse}[teo]{Observation}

\newtheorem{prop}[teo]{Proposition}
\usepackage{manfnt}
\usepackage[bbgreekl]{mathbbol}
\usepackage{stmaryrd}
\usepackage{euscript}
\usepackage{mathrsfs}
\usepackage{txfonts}
\usepackage[all]{xy}
\usepackage{graphics}
\usepackage{soul}
\renewcommand{\Perp}{\bot\!\!\!\bot}
\newcommand{\bbot}{\bot\!\!\!\bot}

\renewcommand{\t}{t}

\newcommand{\IA}{\mathcal{I\!A}}

\newcommand{\cs}{{\scalebox{0.9}{$\mathsf {s}$\,}}}
\newcommand{\ck}{{\scalebox{0.9}{$\mathsf {k}$\,}}}
\newcommand{\ci}{{\scalebox{0.9}{$\mathsf {i}$\,}}}
\newcommand{\ce}{{\scalebox{0.9}{$\mathsf {e}$\,}}}

\newcommand{\cp}{{\scalebox{0.9}{$\mathsf {p}$\,}}}
\newcommand{\ccc}{{\scalebox{0.9}{$\mathsf {c\hspace*{-.18em}c}$\,}}}

\newcommand{\cS}{{\bf \mathsf{\scriptstyle{S}}\,}}
\newcommand{\cK}{{\bf \mathsf{\scriptstyle{K}}\,}}

\newcommand{\cI}{{\bf \mathsf{\scriptstyle{I}}\,}}

\newcommand{\foca}{\ensuremath{\mathcal{{}^FOCA}}}

\newcommand{\rapp}{\mathrm{app}}
\newcommand{\rimp}{\mathrm{imp}}
\newcommand{\imp}{\ensuremath{\operatorname{imp}}}

\newcommand{\rmqp}{\mathrm{QP}}

\newcommand{\vvdash}{\ensuremath{\vdash\!\!\vdash}}

\newcommand{\fapp}{\mathfrak{app}}
\newcommand{\fpush}{\mathfrak{push}}
\begin{document}
\begin{abstract} In this paper we continue with the algebraic study of Krivine's realizability, completing and generalizing some of the
authors' previous constructions by introducing two categories with
objects the abstract Krivine structures and the implicative algebras
respectively. These categories are related by an adjunction whose
existence clarifies many aspects of the theory previously
established. We also revisit, reinterpret and generalize in
categorical terms, some of the results of our previous work such as:
the bullet construction, the equivalence of Krivine's,
Streicher's and bullet triposes and also the fact that these triposes
can be obtained --up to equivalence-- from implicative algebras or
implicative ordered combinatory algebras.
\end{abstract}
\title[The
  category of implicative algebras and realizability] {The category of
  implicative algebras and realizability} \author{Walter Ferrer
  Santos} \address{Facultad de Ciencias\\Udelar\\ Igu\'a
  4225\\11400. Montevideo\\Uruguay\\Centro Universitario de la
  regi\'on este, Udelar\\ Tacuaremb\'o, 20100, Punta del Este,
  Departamento de Maldonado, Uruguay.} \email{wrferrer@cure.edu.uy}
\author{Octavio Malherbe} \address{Facultad de Ingenier\'ia,
  Udelar\\ J. Herrera y Reissig 565 \\ 11300. Montevideo
  \\ Uruguay\\Centro Universitario de la regi\'on este,
  Udelar\\ Tacuaremb\'o, 20100, Punta del Este, Departamento de
  Maldonado, Uruguay.} \email{malherbe@fing.edu.uy} \thanks{The
  authors would like to thank Anii/FCE, EI/UdelaR and LIA/IFUM for
  their partial support. We also thank M. Guillermo and A. Miquel for
  many profitable exchanges on the topics of this paper.} \today
\maketitle
\tableofcontents
\section{Introduction}
\newcounter{bourbaki}
\renewcommand{\thebourbaki}{{\sf{\Roman{bourbaki}}}}
\begin{list}{\large{\sf{\Roman{bourbaki}.}}}{\usecounter{bourbaki}} \bigskip
\item\label{item:intro} In this work, that should be seen as an
  unavoidable completion of \cite{kn:ocar,kn:ocar2}, we explore the
  morphisms of the structures considered therein and analyze Krivine's
  realizability in this perspective. We look at the category with
  objects the \emph{implicative algebras} a.k.a. IAs, and its
  applicative or even computationally dense morphisms and also define
  morphisms of the same kind for \emph{Abstract Krivine Structures}
  a.k.a. AKSes. This is a continuation of the build--up of the
  categorical viewpoint in classical realizability theory as appeared
  firstly in Streicher's work (see \cite{kn:streicher}) and was
  followed by the contributions presented in
  \cite{kn:report,kn:ocar,kn:ocar2}. For the definitions of the
  morphisms that are suggested in this paper, we adapted the basic
  ideas appearing in \cite{kn:OostenZou2016}, related also to the work
  in \cite{kn:hofstra2003} and \cite{kn:hofstra2006}.

\item\label{item:intro2} We defined the morphisms in Sections
  \ref{section:appmorfisms}, \ref{section:moraks} the functoriality of
  the constructions was shown in Section \ref{section:functoriality},
  and the main results appear in Sections \ref{section:adjunction},
  \ref{section:changing}, \ref{section:IAtoTripos} and
  \ref{section:implocas}. 
\item\label{item:description} Next, we proceed to a more precise
  description of each of the Sections.

  \smallskip
  \noindent In Section \ref{section:iaaks}, we recall/introduce, some
  concepts treated in previous papers by the current authors or by
  collaborators --though not always in the same guise-- e.g.
  \cite{kn:ocar,kn:ocar2}, \cite{kn:implimiquel,kn:implimiquel2} and
  also the recent thesis \cite{kn:emthesis} --where an overview of the
  structures treated in this paper is presented. Therein, we recall
  the concept of implicative algebra due to A. Miquel (see
  \cite{kn:implimiquel, kn:implimiquel2, kn:newmiquel} and also
  \cite{kn:emthesis}), and of its twin sister: ordered combinatory
    algebras with full adjunction. We also recall the concept of
    abstract Krivine structure due to Streicher, and two basic
    constructions, named $A$ and $K$, that were introduced in
    \cite{kn:ocar} and \cite{kn:ocar2} and go back and forth between
    abstract Krivine structures and implicative algebras. We also
    give some results related to these concepts that
    are needed later. Moreover, we adapt/recall some basic
    constructions we have developed previously (particularly in
    \cite{kn:ocar2} and that also appeared in \cite{kn:implimiquel2})
    that produce triposes from implicative algebras. We intend to
    understand the results on the equivalence of the triposes that we
    name as Krivine's, Streicher's and bullet tripos (see
    \cite{kn:ocar2}) in terms of the morphisms introduced in IAs and
    AKSes in Section \ref{section:appmorfisms}. The results concerning
    this, appear in Sections \ref{section:IAtoTripos} and
    \ref{section:implocas}.

  \smallskip
  \noindent In Section \ref{section:appmorfisms} we
  introduce the concepts of computationally dense morphism and
  applicative morphism of implicative algebras (or full adjunction
  ordered combinatory algebras), thus defining two categories with
  objects the implicative algebras: ${\bf IAc}$ and ${\bf IA}$. The
  morphisms defined in this section are, as we mentioned before, an
  adaptation to our context of constructions from
  \cite{kn:OostenZou2016} that are also related to
  \cite{kn:hofstra2003} and \cite{kn:hofstra2006}.

  \smallskip
  \noindent In Section \ref{section:moraks} we define the similarly
  named morphisms between abstract Krivine structures to obtain the
  categories ${\bf AKSc}$ and ${\bf AKS}$. In Section
  \ref{section:functoriality} we prove that the constructions named
  $A$ and $K$ can be extended to functors in the corresponding
  categories (see the diagram below).

  \smallskip
  \noindent In Section
  \ref{section:adjunction} we obtain the adjunction result we were
  looking for, as depicted in the diagram: 
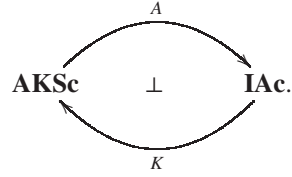
\[\xymatrix{ {\bf AKSc}
    \ar@/^2pc/[rr]^{A} & {\hspace*{-0.19cm}\mbox{\larger[0]$\perp$}} &
    \ar@/^2pc/[ll]^{K}{\bf IAc}.}\]

  \noindent In Section \ref{section:changing} we start by recalling
  some results and definitions on closure operators and their
  \emph{Alexandroff approximations}. Then, we generalize the
  construction appearing in \cite{kn:ocar2} of the map $A_\bullet$
  that produces an IA named $A_\bullet(\mathcal K)$ from an AKS named
  $\mathcal K$ where $A_\bullet(\mathcal K)$ is an implicative algebra
  based on the closed subsets with respect to the \emph{bullet}
  closure operator. The bullet operator is the Alexandroff
  approximation of Streicher's double perpendicular operator. This
  generalization consists in defining a functor called $V: (\mathcal
  A,\iota) \mapsto \mathcal A_\iota$ with domain the category of
  (implicative) comonads on ${\bf IA}$ (named as ${\bf
    Co}_{\operatorname{imp}}({\bf IA})$) and codomain ${\bf
    IA}$. Moreover, $V$ is the right adjoint of the functor $T : {\bf
    IA} \to {\bf Co}_{\operatorname{imp}}({\bf IA})$ that produces the
  trivial comonad. The above-mentioned results are particular cases of 
  categorical constructions that appear for example in R. Street,
  \emph{The formal theory of monads}, \cite{kn:street}. In particular,
  the functor $V$ is the opposite version of the one called therein
  \emph{the construction of algebras} functor that in our case is
  simply the fixed point functor, i.e. the functor that sends the
  implicative algebra based on $A$ in the implicative algebra based on
  $A_{\iota}=\{a \in A: \iota(a)=a\}$.

  \smallskip
  \noindent
  In Section
  \ref{section:IAtoTripos} we revisit the results of our previous
  papers \cite{kn:ocar,kn:ocar2} in order to complete and reformulate
  their main results in categorical terms using the tools introduced
  along the paper. In particular, we prove the theorem that guarantees
  that the triposes associated to $\mathcal A \in {\bf IA}$ and
  $\mathcal A_\iota$ are equivalent as well as the triposes associated
  to $\mathcal B \in {\bf IA}$ and $A(K(\mathcal B))$.

  \smallskip
  \noindent
  In the final
  Section \ref{section:implocas}, we present a way to obtain the main
  constructions of \cite{kn:ocar} and \cite{kn:ocar2}, and more
  specifically the results of Streicher's paper \cite{kn:streicher},
  in terms of the categorical viewpoint taken up in this work. This is
  necessary due to the fact that the structures the author uses
  therein, are not implicative algebras as the implication does not
  commute with infinite meets. Hence, new categories have to be
  introduced (the category of implicative ordered combinatory algebras
  --named ${\bf IOCA}$--) and new aspects need to be taken care of.
  \section{Implicative algebras, abstract Krivine structures}\label{section:iaaks}
  \item In this section, we recall briefly the definitions of
implicative algebras (IAs), FOCAs and AKSes and some constructions
introduced in \cite{kn:ocar2} that will be completed in this paper and
transformed into functors in the adequate categories. For details on
these constructions, we refer the reader to the just mentioned
reference. The class of all IAs will be denoted as $\IA$, of all FOCAs
will be denoted as $\foca$ and of all AKSes as $\mathcal {AKS}$.
\item\label{item:miquel} In this paragraph, we review the basic
  definitions of implicative algebras and implicative structures as
  employed by A. Miquel in \cite{kn:implimiquel} and
  \cite{kn:implimiquel2}.
  \begin{defi}(A. Miquel:
    \cite{kn:implimiquel}, \cite{kn:implimiquel2} and
    \cite{kn:newmiquel}).\label{defi:mainimpl} 
      An \emph{implicative} structure is a triple $\mathcal
      A=(A,\leq,\operatorname{imp})$
      where:
      \begin{enumerate}
      \item \label{item:i1}$(A,\leq)$ is a
        complete meet semilattice;
      \item\label{item:i2}$\imp:A \times A \rightarrow A$ is a map
        also denoted as $a\rightarrow b:=\imp(a,b)$ satisfying:
        \begin{enumerate*}
        \item[(\,i\,)]\label{item:i21} If $a'\leq a$, $b \leq b'$,
          then $a \rightarrow b \leq a'\rightarrow
          b'$,
        \item[(\,ii\,)] \label{item:i22}$\bigcurlywedge_{b \in
            B} (a \rightarrow b) = a \rightarrow \bigcurlywedge_{b\in
            B} b$ for all $B \subseteq
          A$. \end{enumerate*}
      \end{enumerate} We denote: $a\to b \to
      c:=a\to(b \to c)$.
  \end{defi}
  In the definition above, $\bigcurlywedge$ denotes the infimum of the
  corresponding subset of $A$ with respect to the given
  order. Sometimes we use the standard notation $\operatorname{inf}$
  for the infimum instead of $\bigcurlywedge$.
  \begin{defi}
    \begin{enumerate}
    \item \label{item:i4} Given, $a,b \in A$, we define
      $a\cdot_\rightarrow b:= \bigcurlywedge \{c: a \leq (b
      \rightarrow c)\} \in A$. We call {\em application} and denote it as $\operatorname{app}:A \times A \to
      A$, the function given as: $\operatorname{app}(a,b):=
      a\cdot_\rightarrow b$.
    \item \label{item:i5} We call
      \begin{enumerate*}
      \item \label{item:i52}$\ck=\bigcurlywedge_{a,b \in A}(a\rightarrow b \rightarrow a)$,
      \item \label{item:i53}$\cs=\bigcurlywedge_{a,b,c \in
        A}((a\rightarrow b \rightarrow c)\rightarrow (a \rightarrow
        b)\rightarrow a \rightarrow
        c)$,
      \item \label{item:i51}$\ci=\bigcurlywedge_{a \in A}(a\rightarrow
        a)$.
      \end{enumerate*}
    \item \label{item:i6} A \emph{separator} in $\mathcal A$ is a
      subset $\mathcal S \subseteq A$ such
      that:
      \begin{enumerate*}
      \item \label{item:i61} If $a \in
        \mathcal S$ and $a \leq b$ then $b \in \mathcal
        S$,
      \item \label{item:i63}If $a \rightarrow b \in \mathcal S$
        and $a\in \mathcal S$, then $b \in \mathcal S$ (\emph{modus
          ponens})
      \item \label{item:i62} $\ck,\cs \in \mathcal
        S$.
      \end{enumerate*}
    \item \label{item:i7} An {\em implicative
        algebra} is a pair $(\mathcal A,\mathcal S)$ where $\mathcal
      A$ is an implicative structure and $\mathcal S$ is a separator
      in $\mathcal A$. We call $\IA$ the family of implicative
      algebras.
    \item\label{item:i8} We call: $\ccc=\bigcurlywedge(((a \rightarrow
      b)\rightarrow a) \rightarrow a)$, $\perp=\bigcurlywedge_{a \in
        A}a$ and $\top=\bigcurlywedge_{a \in \emptyset}a$. We say that
      the implicative algebra is classical when $\ccc \in \mathcal S$
      and that it is consistent when $\perp \not \in \mathcal
      S$.
    \end{enumerate}
  \end{defi}
  \begin{nota}If the pair
    $(\mathcal A,\mathcal S)$ is an implicative algebra as defined
    above, sometimes we abuse notations and say that $\mathcal A$ is an
    implicative algebra with separator $\mathcal S_{\mathcal
      A}:=\mathcal
    S$.
  \end{nota}
  \begin{obse}\label{obse:initial} \label{obse:preordersentice}
    \begin{enumerate}
      \item It is important to recall that a complete meet semilattice
        can be proved to be also a complete lattice. Following the
        original Miquel's definition, we state explicitly only half of
        the completeness property for reasons of emphasis as the other
        half of the completeness property is rarely used in the
        theory.
    \item The operation $\cdot_\to$ is monotonic in both
      variables; \item For all $a,b,c \in A,\, a \cdot_\to b \leq c
      \text{ if and only if } a \leq (b \rightarrow c)$. i.e.  for all
      $b \in A$, the functor $-\cdot_\to b$ is left adjoint to $b \to
      -$
      \begin{equation}\label{eqn:commutdiag}\xymatrix{ (A,\leq)
          \ar@/^2pc/[rr]^{- \cdot_\to b} & {\mbox{\larger[-2]$\perp$}}
          & \ar@/^2pc/[ll]^{b \to -}(A,\leq),}
      \end{equation}
      and the unit and counit are respectively: $a \leq b \rightarrow
      (a \cdot_\to b)\,\,\textrm{and}\,\,(a \to b)\cdot_\to a \leq
      b$. Notice that from the unit inequality, the fact that
      $\mathcal S$ is upwards closed, and the {\em modus ponens} rule,
      we deduce easily that if $a,b \in \mathcal S$, then $a \cdot_\to
      b \in \mathcal S$.
    \item Later (see Paragraph \ref{item:implvsfoca}: {\bf From
      implicative algebras to FOCAs}.) we prove that the combinators
      $\ck$ and $\cs$ of an implicative algebra, can be characterized
      in terms of the application morphism (written as juxtaposition,
      i.e. $a \cdot_{\to} b = ab$) as:
      \begin{enumerate}
      \item $\ck=\sup\{\ell \in A: \forall a,b \in A, \ell ab \leq a$\};
        \item $\cs=\sup\{\ell \in A:\forall a,b,c \in A, \ell abc \leq
          ac(bc)\}$;
          \item $\ci=\sup\{\ell \in A: \forall a \in A, \ell a \leq
            a$\}.
      \end{enumerate}
    \item For future reference we recall the following definition of a
      preorder $\vdash$ in $A$ that is called the entailment relation
      (see \cite[Definition 4.9]{kn:ocar} and \cite[Definition
        3.13]{kn:ocar2}) : for $a, a' \in A$ we write $a \vdash a'$ if
      $a \to a' \in \mathcal S$. In that case it follows from (3) that
      $(a \to a') \cdot_\to a \leq a'$, and any element $s \in S$ such
      that $s \cdot_\to a \leq a'$ is called a realizer of the
      relation $a \vdash a'$ (and we write $s\vvdash (a \vdash
      a')$). It is clear that the set of realizers of $a
        \vdash a'$ is the set of elements $\{s \in \mathcal S: s \leq
        a \to a'\}$. 
    \item Once we display some other elements of $\mathcal S$, it can
      be proved easily that the entailment relation is transitive (see
      for example \cite[Lemma 4.14, (3)]{kn:ocar} or \cite[Proposition
        3.20]{kn:newmiquel}). \item The reader should be aware that in
      the papers mentioned before we used two notations $a \sqsubseteq
      a'$ and $a \vdash a'$ for the entailment relation, hereafter we
      will only use the second. This preorder depends on the separator
      $\mathcal S$ and when needed we denote it as $\vdash^{\mathcal
        S}$.
\end{enumerate}
\end{obse}
We present here the following simple lemma that is used in Theorem
\ref{theo:mainadj}.
\begin{lema}\label{lema:equalinf} Assume that $\mathcal A$ is an implicative algebra and that $C,D \subseteq A$. Then: $\bigcurlywedge\{x\to
y: x \leq \bigcurlywedge C, y \in D\}= \bigcurlywedge C \to
\bigcurlywedge D$. \end{lema} \begin{proof} Suppose $D\neq\emptyset$ and consider the sets $A=
  \{x\to y: x \leq \bigcurlywedge C, y \in D\}$ and $B=
  \{\bigcurlywedge C\to y: y \in D\}$. Since for every $a\in A$ there
  exists a $b\in B$ such that $b\leq a$ we have that: $\bigcurlywedge
  B\leq \bigcurlywedge A$ ($\bigcurlywedge B$ is a lower bound of
  $A$), but $\bigcurlywedge B=\bigcurlywedge C\to \bigcurlywedge D$ by
  preservation of infimum (see Definition \ref{defi:mainimpl}\,
  (\ref{item:i22})). On the other hand, since $B\subseteq A$ then
  $\bigcurlywedge B\geq \bigcurlywedge A$, i.e.  $\bigcurlywedge C\to
  \bigcurlywedge D\geq \bigcurlywedge A$. If $D=\emptyset$ we have
  $\bigcurlywedge\{x\to y: x \leq \bigcurlywedge C, y \in D\}=\bigcurlywedge_{a \in \emptyset}a=\top$ and $\bigcurlywedge C \to
\bigcurlywedge D=\bigcurlywedge(\bigcurlywedge C \to
 D)=\top$ by Definition~\ref{defi:mainimpl} (\ref{item:i2}).
\end{proof}

\item
\label{item:examplesgen} In \cite{kn:implimiquel,kn:implimiquel2,kn:newmiquel} and in the recent thesis \cite{kn:emthesis} the authors
present examples and applications of the concept of implicative
structures or algebras, such as complete boolean algebras or more
generally complete Heyting algebras; the set\, $\mathcal P(P)$ when
$P$ is a total combinatory algebra; the so called dummy structures
with implications $a \to b:=\top$ or $a \to b:=b$. Also, some
particular classes of implicative algebras appear in
\cite{kn:emthesis} called disjunctive algebras and conjunctive
algebras. In what follows we concentrate in the examples coming from
classical realizability theory.
\item \label{item:implvsfoca}In \cite{kn:ocar,kn:ocar2} the concept of
  ordered combinatory algebra was proposed as the basic foundation of
  realizability. In particular, the concept of
  \emph{full adjunction ordered combinatory algebra} a.k.a. FOCA, is
  presented and the associated tripos is studied. It turns out that
  the concept of implicative algebra and the concept of FOCA are
  basically coincident. This was observed in \cite[Section
    3]{kn:ocar2} and also in the references
  \cite{kn:implimiquel,kn:implimiquel2,kn:newmiquel} in all cases
  without detailed proofs. We recall briefly the relevant definitions
  and write down a complete proof.
  \begin{defi}\label{item:example}\label{example:ocas} A
    \emph{full adjunction combinatory algebra} $\mathcal
    A=(A,\leq,\operatorname{app},\imp,\cs,\ck,\Phi_{\mathcal A})$ is a
    septuple where $(A,\leq)$ is a complete inf-semilattice, and
    $\operatorname{app},\imp$ are maps from $A\times A$ into $A$, the
    second satisfies the monotony condition (\ref{item:i21})(i) in
    Definition \ref{defi:mainimpl} and the first is monotonic in both
    variables (we abbreviate $\operatorname{app}(a,b)=ab$ and
    $\operatorname{imp}(a,b)=a\to b$). Both maps together, satisfy the
    full adjunction property, i.e.  $ab:=\operatorname{app}(a,b) \leq
    c \Leftrightarrow a \leq (b \rightarrow c):= \imp(b,c)$. The
    \emph{filter} $\Phi_{\mathcal A}$ is a subset of $A$ that is
    closed under application, is upper closed and contains the
    elements $\cs$ and $\ck$. Moreover, $\cs$ and $\ck$ are elements
    with the following properties: $\ck ab \leq a$ and $\cs abc \leq
    (ac)(bc)$.
  \end{defi}
  \begin{nota} If $\mathcal A$ is a FOCA as
    above, the filter will be frequently abbreviated as
    $\Phi$. Moreover, the element $\ci=\cs\ck\ck$ is in fhe filter and
    it is easy to show that for all $a \in A$, $\ci a = \cs\ck\ck a
    \leq (\ck a)(\ck a)\leq a$.
  \end{nota}
  \begin{description} \item [From FOCAs to implicative algebras]
    \begin{enumerate}
    \item Observe those general considerations about adjoint functors
      (or else Observation \ref{obse:initial} and a direct argument
      --see for example \cite[Section 3]{kn:ocar}) can be used in
      order to deduce that the implication of a FOCA commutes with
      infinite meets, i.e. the full condition \eqref{item:i22} in
      Definition \ref{defi:mainimpl} holds in this
      situation. Moreover, it is easy to prove --compare with
      \cite{kn:ocar2}-- that for a FOCA one has that $ab=a\cdot_\to
      b$. Hence, given the FOCA we have the first ingredients of an
      IA, i.e. a complete meet semilattice and a map --the implication
      of the FOCA-- satisfying the adequate properties.
\item Next, we show that the filter of the FOCA --named $\Phi$-- is a
  separator.  The fact that it is upper closed by inequalities is part
  of the definition of a filter. To check the \emph{modus ponens} we
  proceed as follows: take $a \in \Phi$ and $a \to b \in \Phi$, we can
  use the closedness of $\Phi$ by the application map, to deduce that
  $(a \to b)\cdot_\to a \in \Phi$ and hence as $(a \to b)a \leq b$
  (see the inequality for the counit of the adjunction), we conclude
  that $b \in \Phi$. The proof that $\bigcurlywedge_{a,b \in
    A}(a\rightarrow b \rightarrow a),\,\bigcurlywedge_{a,b,c \in
    A}((a\rightarrow b \rightarrow c)\rightarrow (a \rightarrow
  b)\rightarrow a \rightarrow c) \in \Phi$, follows directly by
  application of the full adjunction between application and
  implication for a FOCA. For example, if we apply the inequality $\cs
  xyz \leq (xz)(yz)$ (valid for all elements of $A$) to the case that
  $x=a\to b \to c,\,y=a\to b,\,z=a$ we deduce that $\cs (a\to b \to
  c)(a\to b)a \leq ((a\to b \to c)a)((a\to b)a) \leq (b \to c)b \leq
  c$, where the last two inequalities come from applying the counit of
  the adjunction between application and implication. Hence as the
  above inequality holds for all $a,b,c \in A$, we deduce that $\cs
  \leq \bigcurlywedge_{a,b,c \in A}((a\rightarrow b \rightarrow
  c)\rightarrow (a \rightarrow b)\rightarrow a \rightarrow c)$ and
  then that the written infimum is an element of the separator. For
  $\ck$ one proceeds similarly.
     \end{enumerate}
  \item [From implicative algebras to FOCAs] Given and implicative
    algebra as above with a separator $\mathcal S$, we define a FOCA
    that has the following structure: the basic meet semilattice is
    the same, the operations are $\operatorname{app}:=\cdot_\to$, the
    implication and $\ck$ and $\cs$ are the same, and the filter
    coincides with the separator, i.e. $\Phi:=\mathcal S$. It will be
    clear that the above construction produces a FOCA once we prove
    the following three assertions that guarantee that $\ck$ and $\cs$
    as well as $\mathcal S$ satisfy the required FOCA's
    conditions. Taking into account that $\cdot_\to=\operatorname{app}$ we write $ab$ in this context, instead of $a\cdot_\to b$ we just write $ab$.
    \begin{enumerate} \item As $\ck \leq a\to(b \to a)$
      for all $a,b \in A$, using the full adjunction we have that $\ck
      a \leq b \to a$ and $\ck a b \leq a$. \item For all $a,b,c \in
      A$ we have that $\cs \leq (a \rightarrow b \rightarrow
      c)\rightarrow (a \rightarrow b)\rightarrow a \rightarrow c$ and
      we need to prove that for all $a,b,c \in A$ we have: $\cs abc
      \leq (ac)(bc)$.  The proof below is extracted from
      \cite{kn:implimiquel2} and uses systematically the inequality
      (unit of the adjunction) valid for all $d,e \in A$, $d \leq e
      \to de$. By definition $\cs=\bigcurlywedge_{x,y,z \in A}((x \to
      y \to z)\to (x \to y) \to x \to z)\leq \bigcurlywedge_{a,b,c \in
        A}((c \to bc \to (ac)(bc))\to (c \to bc) \to c \to
      (ac)(bc)))$. This last inequality is a consequence that the
      change of variables, $x=c, y=bc, z=(ac)(bc)$ implies that the
      infimum is taken over a set that a priori is contained in the
      other one. Now, as $bc \to (ac)(bc) \geq ac\,,\,c \to ac \geq
      a\,,\, c \to bc \geq b$ and using the fact that the implication
      is contravariant in the first variable we deduce that
      $\bigcurlywedge_{a,b,c \in A}((c \to bc \to (ac)(bc))\to (c \to
      bc) \to c \to (ac)(bc)))\leq \bigcurlywedge_{a,b,c \in A} a \to
      b \to c \to (ac)(bc)$. We have proved that for all $a,b,c \in A$
      we have that: $\cs \leq (a\to b\to c \to (ac)(bc))$ so that $\cs
      abc \leq (ac)(bc)$. \item If $a,b \in \Phi=\mathcal S$, then as
      $ab \in \Phi=\mathcal S$ as was proved in Observation
      \ref{obse:initial}.
  \end{enumerate} \end{description}
  Notice that putting together the results on the combinators proved
  above, we obtain a proof of the characterization of $\ck$ and $\cs$
  of a FOCA --or an implicative algebra-- written in
  \ref{obse:initial},(4). The proof of the corresponding property for
  $\ci$ is left to the reader to prove.
\item \label{item:AKconstruction}\label{item:example1} Next we recall
  the definitions of AKSes and of the maps $A$ and $K$ depicted in the
  diagram below (see
  \cite{kn:ocar2}): \begin{equation}\label{eqn:mainadj}\xymatrix{
      \mathcal{AKS} \ar@/^2pc/[rr]^{A} & & \ar@/^2pc/[ll]^{K}
      \mathcal{I\!A}.} \end{equation}
  \begin{description}
  \item[Definition of AKS] The definition of an AKS that follows, is
    different --but for our purposes equivalent-- to the one
    introduced initially in \cite{kn:streicher} and later in our
    previous work (see \cite{kn:ocar} and
    \cite{kn:ocar2}).
    \begin{enumerate} \item An abstract Krivine
      structure, AKS is a septuple $\mathcal
      K=(\Pi,\fapp,\fpush,\bbot,\mathrm{QP},\cK,\cS)$ with $\Pi$ an arbitrary
      set; $\bbot$ a subset of $\Pi \times \Pi$ called \emph{the
        pole}; $\fapp,\fpush: \Pi\times \Pi \to \Pi$ are two maps
      called \emph{application} and \emph{push} and written also as
      $\fpush(t,\pi)=t \cdot \pi$ and $\fapp(t,s)=ts$. We denote $t
      \perp \pi$ when $(t,\pi) \in \bbot$. The subset $\bbot \subseteq
      \Pi \times \Pi$ is assumed to be \emph{compatible} with the maps
      $(\fapp,\fpush)$ in the sense that for all $t,s,\pi \in \Pi$,\,
      $t \perp \fpush(s,\pi) \Rightarrow \fapp(t,s) \perp \pi$\,,
      (i.e., $t \perp s\cdot \pi \Rightarrow ts\perp \pi$). The set
      $\mathrm {QP} \subseteq \Pi$, called the set of
      \emph{quasi--proofs}, satisfies that:
      $\fapp(\mathrm{QP},\mathrm{QP})=(\mathrm{QP})(\mathrm{QP})
      \subseteq \mathrm{QP} \subseteq \Pi$, and $\cK,\cS \in
      \mathrm{QP} \subseteq \Pi$.  Moreover,
      \begin{enumerate}
      \item For all $t,s,\pi \in \Pi$, if $t \perp \pi$ then $\cK
        \perp t\cdot s\cdot\pi$;
      \item For all $t,s,u, \pi \in \Pi$ ,
        if $tu(su) \perp \pi$ then $\cS \perp t\cdot s\cdot
        u\cdot\pi$.
        \end{enumerate}
    \item The AKS is said to be \emph{strong} if for all $t,s,\pi \in
      \Pi$ the following holds: $ t \perp \fpush(s,\pi)
      \Leftrightarrow \fapp(t,s) \perp \pi $ (i.e., $t \perp s\cdot
      \pi \Leftrightarrow ts\perp \pi$).
    \item We use the following
      convention regarding latin and greek letters: the first variable
      in the map $\fpush$ is a latin letter and the
      second is a greek letter, and both variables are latin letters
      for the map $\fapp$. When an element of $\Pi$ appears
      in the left side of the relation $\bbot$, it will be denoted
      with a latin letter and it will be denoted with a greek letter
      when it appears at the right side, e.g. $t \perp
      \pi$\footnote{This unortodox notation is helpful because we have
        simplified the definition of AKS by identifying the sets of
        terms ($\Lambda$) and of stacks ($\Pi$) --that in fact are
        denoted and named as such in Krivine's work.  In order to keep
        the intuition we write $t,s,\dots$ when the element of $\Pi$
        is understood as a term and $\pi,\sigma,\cdots$ when it is
        viewed as a stack.}.  For subsets of $\Pi$ we adopt the
      following notations, we use capital letters $L,M,N$ when the set
      appears at the left of a perpendicularity symbol $\perp$ and
      $P,Q,R$ when appears at the right. For example, when we write $t
      \perp P$ we mean that $t \perp \pi$ for all $\pi \in P$ and
      similarly for $L \perp \pi$ that means: $\ell \perp \pi$ for all
      $\ell \in L$. We define ${}^\perp P=\{\t \in \Pi: t \perp P\}$
      and likewise for $L^\perp$. We can also define a closure
      operator (see Definition \ref{defi:appendix}) on $\mathcal
      P(\Pi)$ as follows: $P \leadsto \overline{P}: \mathcal P(\Pi)
      \to \mathcal P(\Pi)$ with $\overline{P}=({}^\perp P)^\perp$ (see
      \cite{kn:ocar,kn:ocar2} and \cite{kn:streicher}). For future
      reference we recall that --see \cite{kn:ocar2}-- there is also
      an Alexandroff closure operator (see Definition
      \ref{defi:appendix}) $(-)^\wedge$ on $\mathcal P(\Pi)$
      associated to the polarity defined as follows:
      $P^\wedge=\bigcup\{\overline{\{\pi\}}: \pi \in P\}$. 
\end{enumerate}
  
\item[From abstract Krivine structures to implicative
  algebras]\label{item:akstoia} Next we describe briefly the
  construction of the map $A$, $A:\mathcal{AKS} \to \IA$, see
  \cite{kn:ocar2}, and item \eqref{item:eseka} below.
  \begin{enumerate}
  \item Consider $\mathcal K$ an AKS and take $A:=\mathcal P(\Pi)$
    (the power set of $\Pi$) endowed with the reverse inclusion order,
    i.e., $P \leq Q$ if and only if $Q \subseteq P$. It is clear that
    $(A,\leq)$ is an $\operatorname{inf}$-complete semilattice with
    the infimum given by: $\bigcurlywedge\{P_i: i \in
    I\}=\bigcup\{P_i: i \in I\}$. \item The original push map of
    $\mathcal K$, $\Pi \times \Pi
    \stackrel{\mathfrak{push}}{\longrightarrow} \Pi$ induces naturally
    a map on $\mathcal P(\Pi)$ --that we denote as
    $\operatorname{push}:\mathcal P(\Pi) \times \mathcal P(\Pi) \to
    \mathcal P(\Pi)$. It is clear that, in the nomenclature of
    Definition \ref{defi:mainimpl}, \eqref{item:i2}, the map
    $\to:=\operatorname{push}({}^\perp(-) \times
    \operatorname{id}):\mathcal P(\Pi) \times \mathcal P(\Pi) \to
    \mathcal P(\Pi)$ or more explicitly $(P,Q) \mapsto (P\to
    Q)=\{t\cdot\pi: t\perp P\,,\, \pi \in Q\}={}^\perp P\cdot Q$, is
    an \emph{implication} in the inf-complete semilattice $(\mathcal
    P(\Pi),\supseteq)$, i.e. $(\mathcal P(\Pi),\supseteq,\to)$ is an
    implicative structure.
 \item It follows immediately from this definition that the
    associated application map $\cdot_\to:\mathcal P(\Pi)
    \times\mathcal P(\Pi) \to \mathcal P(\Pi)$ is given, by: $P
    \cdot_\to Q=\big\{\pi \in \Pi: \forall t \perp Q\,,\, t \cdot \pi
    \in P\}$.
  \item \label{item:eseka} If we call $\mathcal S_{\mathcal K}=\{P\in
    \mathcal P(\Pi): \exists t \in \mathrm{QP}, t \perp P\}=\{P
    \subseteq \Pi: {}^\perp P \cap \mathrm{QP}\neq \text{\O}\}$, it is
    clear that the subsets $\{\cK\}^{\perp}=\{\pi \in \Pi: \cK \perp
    \pi\}, \{\cS\}^{\perp}=\{\pi \in \Pi: \cS\perp \pi\} \subseteq
    \Pi$ are in $\mathcal S_{\mathcal K}$. Moreover if we construct in
    the manner of Definition \ref{defi:mainimpl}, the subsets
    $\ck^{\operatorname{IA}},\cs^{\operatorname{IA}}$ as infinite
    meets (in our case infinite unions in $\Pi$) of the adequate
    families of subsets, it can be proved --using the notations
    above-- that $\ck^{\operatorname{IA}} \subseteq
    \{\cK\}^\perp,\cs^{\operatorname{IA}} \subseteq \{\cS\}^\perp$,
    and consequently $\ck^{\operatorname{IA}},\cs^{\operatorname{IA}}
    \in \mathcal S_{\mathcal K}$. The rest of the proof that $\mathcal
    S_{\mathcal K}$ is a separator --and the missing details above-
    can be found for example in \cite[Lemma 2.37, Theorem
      2.38]{kn:newmiquel} or in \cite[Section 6]{kn:ocar2}. In the
    mentioned references the function $A$ is defined as $A(\mathcal
    K):=(\mathcal P(\Pi),\supseteq,\to,\mathcal S_{\mathcal K})$.
  \end{enumerate}
\item[From implicative algebras to abstract Krivine
  structures] \label{item:main2} In the opposite direction to the
  construction of the map $A:\mathcal{AKS} \to \IA$, in
  \cite[Definition 5.12]{kn:ocar2} we considered a map $K:\IA \to
  \mathcal{AKS}$ that we recall here. Given an implicative algebra:
  $\mathcal A=(A,\leq,\rimp_A=\to_A,\mathcal S)$ --call
  $\mathrm{app_A}=\cdot_\rightarrow$ its associated application
  morphism and ${\ck},{\cs}$ the basic combinators. Define $K(\mathcal
  A)=(\Pi,\Perp,\mathfrak{app}, \mathfrak{push}, \cK,\cS,\mathrm{QP})$
  as:
  \begin{enumerate}
  \item $\Pi:=A$;
  \item $\Perp:=\,\leq$\,,\, i.e.\,\, $s\perp\pi:\Leftrightarrow
    s\leq\pi$;
  \item
    $\mathfrak{app}:=\rapp_A$ \quad,\quad $\mathfrak{push}:=\rimp_A$;
  \item $\cK := {\ck}\quad ,\quad \cS :={\cs}$;
  \item $\mathrm{QP}:=\mathcal S$.
  \end{enumerate}

  In \cite{kn:ocar} and \cite{kn:ocar2} it was proved that $K(\mathcal A)$ is an AKS (in the equivalent context of FOCAs).
\end{description} \item \label{item:koca} In this paragraph, along which we follow closely Miquel's \cite[Section: The implicative
tripos]{kn:implimiquel2,kn:newmiquel}, we briefly show how to adapt
  some constructions appearing in \cite{kn:ocar,kn:ocar2}
  --specifically the construction of an indexed Heyting preorder from
  an implicative ordered combinatory algebra-- to the context of
  implicative algebras.  This is necessary in order to set up an 
  adequate platform to prove the results of Section
  \ref{section:IAtoTripos}. For details in the definition of Heyting
  preorders --that is not a standard term\footnote{As one of the reviewers suggest, probably a better name for the object defined is: {\em a cartesian closed preorder}.}-- we refer the reader to
  \cite[Definition 4.7]{kn:ocar} and to \cite[Remarks 4.8]{kn:ocar}
  for a discussion on this nomenclature. A Heyting preorder is a
  preorder $(D,\leq)$ equipped  with two binary operations $\wedge, \to$
  and a distinguished   element $\top$. The first operation $\wedge$
  satisfies the usual conditions for a ``meet'' and the second ``the
  Heyting implication'' satisfies that for all elements of $D$, $a
  \wedge b \leq c$ if and only if $a \leq b \to c$.  We call ${\bf
    Ord}$ the category of preorders and monotonic maps and
  ${\bf{HPO}}$, the category of Heyting preorders with morphisms the
  monotonic maps $f:(A,\leq)\to (B,\leq)$ such that
  $f(\top)\cong\top$, $f(a\wedge b) \cong f(a)\wedge f(b)$, $f(a\to
  b)\cong f(a)\to f(b)$ for all $a,b\in A$. An indexed HPO is a pseudo
  functor ${\bf F}:{\bf Set}^{\operatorname{op}}\to {\bf HPO}$ and
  similarly for the definition of indexed preorders. A particular case
  of indexed Heyting preorders are the triposes. For a list of the
  necessary additional conditions we refer the reader for example to
  \cite[Definition 2.9]{kn:ocar2}.

The two following assertions will be used
later. \begin{rema}\label{remark:equivalencee} \begin{enumerate} \item
    An indexed monotonic map $\sigma:{\bf C}\to{\bf D}$ of indexed
    preorders (see definition above) is an equivalence, if and only if for every set $I$, the
    monotonic map $\sigma_I:{\bf C}(I)\to{\bf D}(I)$ is order
    reflecting and essentially surjective. \item If ${\bf C}$ and
    ${\bf P}$ are equivalent indexed preorders and ${\bf P}$ is a
    tripos, then so is ${\bf C}$, and as triposes are equivalent.
\end{enumerate} \end{rema} \begin{description} \item[From implicative algebras to Heyting preorders] If $\mathcal A$ is an implicative
algebra based on a set $A$ with separator $\mathcal S$ and maximum
element $\top\in\mathcal S$ define $\mathcal H(\mathcal A):=(A,
\vdash, \wedge, \rightarrow)$ where: $\vdash$ is the relation defined
in Observation \ref{obse:preordersentice} --see the properties
therein--, the map $\to:A\times A \to A$ is the implication of
$\mathcal A$, and the map $\wedge: A \times A \rightarrow A$ is $a
\wedge b:= {\cp} ab$ (recall that any implicative algebra is an
ordered combinatory algebra, see \cite{kn:implimiquel}, and that
$\cp:=\lambda x \lambda y \lambda z. zxy$). Then $\mathcal H(\mathcal
A)$ is a Heyting preorder. The proof of this assertion follows closely
the proof in the case of implicative ordered combinatory
algebras. Regarding the proofs in this case we refer the reader to
\cite[Observation 3.9]{kn:ocar2} for the basic properties of $\cp$; to
\cite[Section 4, Lemma 4.14]{kn:ocar} for the proof that the above
construction produces a meet-semilattice and to \cite[Section 4,
  Theorem 4.15]{kn:ocar} for the rest of the proof. A careful
consideration of the relationship between Heyting preorders and
implicative algebras is a main theme in
\cite{kn:newmiquel}. \item[Products of implicative algebras and
  indexed Heyting preorders] Assume that $\mathcal A$ is an
implicative structure and consider the implicative structure $\mathcal
A^I$ for a fixed set of indexes $I$ that has as basic set the product
$A^I$ and as order and implication morphism the cartesian product of
the order of $A$ and of the implication of $A$. If $\mathcal S$ is a
separator for $\mathcal A$ the set $\mathcal S[I]=\{\varphi:I
\to A: \bigcurlywedge_{i \in I}\varphi(i) \in S\}$ is a separator for
$\mathcal A^I$ called the {\em uniform power separator} (see \cite[Section 4.2]{kn:newmiquel}).

\begin{defi}\label{defi:deltasubi} The implicative
  algebra defined as the implicative structure $\mathcal A^I$ endowed
  with the separator $\mathcal S[I]$ is denoted as $\mathcal A
  [I]$. \end{defi}

It is clear that the preorder $\vdash^{\mathcal S[I]}$ in $A^I$
associated to the separator $\mathcal S[I]$ is the following: for
$\varphi,\psi:I \to A$, $\varphi \vdash^{\mathcal S[I]} \psi$ if and
only if $\bigcurlywedge\{\varphi(i)\to \psi(i):i \in I\} \in \mathcal
S$.

\begin{defi} In $A^I$ the preorder $\vdash^{\mathcal S[I]}
  \subseteq A^I \times A^I$ will be denoted as
  $\vdash_I:=\vdash^{\mathcal S[I]}$ and called the entailment
  relation, or the entailment
  preorder.
\end{defi}

\begin{obse}\label{obse:delicatenotation} In
  the situation above, when it is cumbersome to write all the time in
  the formulae the description of the propositions $\varphi, \psi$ as
  maps, we omit it in the notations, and write $\varphi(i) \vdash_I
  \psi(i):=\varphi \vdash_I \psi$ (see for example Definition
  \ref{defi:morfimplalgsimpler} and Observation
  \ref{obse:uniform}). Explicitly $\varphi(i) \vdash_I \psi(i)$ if and
  only $\bigcurlywedge_{i \in I}\varphi(i) \to \psi(i) \in \mathcal
  S$.
\end{obse}

We consider for $I$ a fixed set the following maps:
a) $[I]$ as map from the class $\IA$ into itself as in Definition
\ref{defi:deltasubi}, i.e. $[I](\mathcal A):=\mathcal A[I]$; b)
$\mathcal H$ the construction of a Heyting preorder as described above
; c) ${\bf H}_{\mathcal A}(I)$ the composition of the two maps just
defined:
\[\xymatrix{\mathcal{I\!A} \ar[rd]_-{{\bf H}_{\mathcal
      A}(I)}\ar[rr]^{[I]}&&\mathcal{I\!A}\ar[ld]^-{\mathcal H}\\&{\bf
    HPO}&}\]

\begin{defi}\label{defi:indexedhpo} Let $\mathcal A$ be
  an implicative algebra and define the indexed HPO, ${\bf
    H}_{\mathcal A}:{\bf Set}^{\operatorname{op}}\to{\bf HPO}$ as follows:
\begin{enumerate} \item ${\bf H}_{\mathcal A}(I):=\mathcal H(\mathcal A[I])$, or more explicitly: ${\bf H}_{\mathcal
A}(I)=(A^I,\vdash_I,\wedge,\to)$; \item If $f:J\to I$, ${\bf
    H}_{\mathcal A}(f)$, that is customary to abbreviate it as ${\bf
    H}_{\mathcal A}(f)=f^*:{\bf H}_{\mathcal A}(I) \to {\bf
    H}_{\mathcal A}(J)$, is defined as $f^*:(A^I,\vdash_I)\to
  (A^J,\vdash_J),\,\, f^*( \varphi) = \varphi \circ f$, for $\varphi
  \in A^I$.
\end{enumerate}
\end{defi}

In what follows when there is
no danger of confusion we write $\varphi f$ instead of $\varphi \circ
f$ for the composition.

The following theorem is a restatement of 
\cite[Section 5, Theorem 5.8]{kn:ocar} in the language of IAs. In the present formulation for
IAs, appears in \cite[Theorem (Associated
  tripos)]{kn:implimiquel}.

\begin{theo}\label{theo:oca-to-islat} Let
  ${\mathcal A}= (A,\leq,\rimp,\mathcal S)$ be an implicative algebra,
  then ${\bf H}_{\mathcal A}$ is a tripos.
\end{theo}
\end{description}

\section{Morphisms of implicative algebras}
\label{section:appmorfisms} \item \label{item:appmorfisms} In a recent
publication van Oosten and Zou (see \cite{kn:OostenZou2016}) gave an
explicit characterization of Hofstra and van Oosten's morphisms
between ordered partial combinatory algebras OPCAs as presented in
\cite{kn:hofstra2003} and \cite{kn:hofstra2006}. Here we adapt these
ideas to the specific case of FOCAs i.e. implicative algebras.

\begin{defi}\label{defi:morfimplalgsimpler} 

Let $(A,\to,\mathcal S_{\mathcal A})$ y $(B,\to, \mathcal S_{\mathcal
  B})$ be two implicative algebras, an \emph{applicative morphism} is
a function $f:A\rightarrow B$ such that: \begin{enumerate} \item
  $f(\mathcal S_{\mathcal A})\subseteq \mathcal S_{\mathcal B}$; \item
  If $I=\{(a,a'): a \vdash a'\} \subseteq A \times A$ then
  $\varphi:I\to B$, $\psi:I\to B$ defined by $\varphi(a,a'):=f(a\to
  a')$, $\psi(a,a'):=f(a)\to f(a')$ satisfy the relation $\varphi
  \vdash_I \psi$ \item For all $P \subseteq A$ we have that
  $f(\bigcurlywedge_A P)=\bigcurlywedge_B f(P)$,
  i.e. $f\left(\bigcurlywedge_A\{x: x \in P\}\right)=\bigcurlywedge_B
  \{f(x): x \in P\}$. \end{enumerate}
\end{defi}

\begin{obse}\label{obse:uniform} Concerning the above definition the  following should be noticed. \begin{enumerate} \item It is well known
that if $f$ preserves binary meets, then it is monotonic, hence an
applicative morphism is monotonic; \item It is clear that (2) in the
Definition \ref{defi:morfimplalgsimpler} is equivalent to the
following assertion: there is an $r \in \mathcal S_{\mathcal B}$:
$\forall a,a' \in A$ such that $a \to a' \in \mathcal S_{\mathcal A}$,
then $r \leq f(a \to a') \to f(a) \to f(a')$. Equivalently, we can
write $f(a \to a')\vdash_I f(a)\to f(a')$ where $I$ is as above; \item
In the above condition (3) $f(P):= \mathcal P(f)(P)$ where $\mathcal
P(f)$ is the function induced by $f$ in $\mathcal P(A)$, this abuse of
notation will be frequent. \end{enumerate} \end{obse} We recall the
following result (see \cite[Lemma
  3.6]{kn:ocar}) that clearly is valid also for IAs as stated below.
\begin{lema}\label{lema:asocas} Let $\mathcal A$ be an implicative algebra with separator $\mathcal S$. Then there is an element $\nu \in \mathcal S$ such that $\forall\, a,b,c \in A\,,\,vabc\leq
  a(bc)$.
\end{lema} The following formulation of the second condition
of the definition of applicative morphism is sometimes convenient as
it is expressed solely in terms of the application map as it was
originally introduced in \cite{kn:OostenZou2016}.
\begin{lema} \label{lema:morfimplalgsimpler} In the situation of Definition \ref{defi:morfimplalgsimpler}, if $f:A \to B$ is a function that
satisfies conditions (1) and (3), it also satisfies condition (2) if
and only if the following condition holds: \[\exists t \in \mathcal
S_{\mathcal B}: \forall s\in \mathcal S_{\mathcal A}, a\in A,
tf(s)f(a)\leq f(sa).\quad\quad (2')\]
  
Moreover, if $\varphi,\psi:I \to A$ are such that $\varphi \vdash_I
\psi$, and $f:A \to B$ is an applicative morphism, then $f \circ
\varphi \vdash_I f \circ
\psi$. \end{lema}

\item \label{item:comupdense} Now we consider the notion of
  \emph{computationally dense} morphism. It turns out that these are
  the morphisms that give rise to geometric morphisms between
  the associated triposes (see \cite{kn:hofstra2003} and
  \cite{kn:OostenZou2016}). For a family of examples of
  computationally dense morphisms the reader can look at Theorem
  \ref{ejemplosdecompsense}.
  
\begin{defi}\label{defi:compdensemor} Let $\mathcal A, \mathcal B$ be implicative algebras with separators $\mathcal S_{\mathcal A}$ and
$\mathcal S_{\mathcal B}$. Let $f:A\rightarrow B$ be an applicative
  morphism. We say that $f$ is \emph{computationally dense} if there
  is a monotonic function $h:\mathcal S_{\mathcal B}\rightarrow
  \mathcal S_{\mathcal A}$, $f \circ h \vdash_{\mathcal S_{\mathcal
      B}} \operatorname{id}_{\mathcal S_{\mathcal B}}$. \end{defi}

      \begin{obse}\label{obse:compdensemor} \begin{enumerate} \item
    Notice that we have taken in the above definition $\mathcal
    S_{\mathcal B}$ as the index set but this is exactly what is
    needed as is explained below.
  \item In explicit terms,  the second condition of this
    definition reads as: $\exists t \in \mathcal S_{\mathcal B} :
    \forall b \in \mathcal S_{\mathcal B},\, tf(h(b)) \leq b$.
\item Even though in the definition the map $h$ is from $\mathcal
  S_{\mathcal B}$ into $\mathcal S_{\mathcal A}$, it is harmless to
  consider it as a map $h:B \to A$ with the additional property that
  $h(\mathcal S_{\mathcal B}) \subseteq \mathcal S_{\mathcal A}$. Of
  course, the monotonicity only holds when restricted to $\mathcal
  S_{\mathcal B}$.
  \end{enumerate}
\end{obse}
\begin{lema} \label{lema:catimpl}
  Let $\mathcal A \stackrel{f}{\rightarrow} \mathcal B
  \stackrel{g}{\rightarrow} \mathcal C$ be implicative algebras and
  $f\,,\,g$ monotonic maps between the underlying
  sets.
  \begin{enumerate} \item If $f$ and $g$ are applicative
    morphisms so is $g f: \mathcal A \to \mathcal C$.
\item If $f$ and $g$ are computationally dense morphisms so is $g f:
  \mathcal A \to \mathcal C$. \item The map
  $\operatorname{id}_\mathcal A:=\operatorname{id}_A:\mathcal A \to
  \mathcal A$ is a computationally dense
  morphism.
  \end{enumerate}
\end{lema}
\begin{proof} The proof is
  direct even though some details must be filled in. For example, for
  the proof that for $a \vdash a'$, then $(gf)(a \to a')\vdash_I
  (gf)(a)\to (gf)(a')$ if we start with $f(a \to a')\vdash_I f(a)\to
  f(a')$ then apply $g$ and then use Lemma
  \ref{lema:morfimplalgsimpler}, we deduce that $g(f(a \to
  a'))\vdash_I g(f(a)\to f(a'))$. Once this is guaranteed it is clear
  that the rest of the proof of this assertion follows
  directly. Similarly, if $h, k$ are maps that satisfy the property of
  Definition \ref{defi:compdensemor} for $f$ and $g$ respectively, we
  prove by a direct application of Lemma
  \ref{lema:morfimplalgsimpler}, that $hk$ is a map that does the job
  for $gf$.
\end{proof}

In accordance with the results just proved,
the following definition is adequate.

\begin{defi}\label{defi:ralgcat1} We define the following
  categories.
  \begin{enumerate}
  \item ${\bf IA}$ that has
    implicative algebras as objects and applicative morphisms as
    arrows. 
  \item ${\bf IAc}$ that is the same as above
    but with computationally dense morphisms as
    arrows.
  \end{enumerate}
\end{defi}

\section{Morphisms of abstract Krivine structures}
\label{section:moraks}
\item The definition of the morphisms $\mathfrak f:\mathcal K \to
  \mathcal K'$ where $\mathcal K, \mathcal K' \in \mathcal{AKS}$, will
  be guided by the definitions introduced in Section
  \ref{section:appmorfisms}. A morphism $\mathfrak f$ will be a set
  theoretical function $f:\Pi\to \Pi'$ with the property that the
  induced map $\mathcal{P}(f):\mathcal{P}(\Pi)\rightarrow
  \mathcal{P}(\Pi') \,,\,\mathcal{P}(f)(P)=f(P)$, is a morphism of
  implicative algebras in the sense considered
  above. \begin{defi} \label{defi:moraks}Let $\mathcal
    K=(\Pi,\operatorname{push}, \bot\!\!\!\bot,
    \mathfrak{app},\mathrm{QP}, \cK,\cS)\,,\,\mathcal K'=(\Pi',
    \operatorname{push}', \bbot',\mathfrak{app}',
    \mathrm{QP}',\cK',\cS')\in
    \mathcal{AKS}$. \begin{enumerate} \item\label{defapplicativeaks}
      An \emph{applicative morphism} $\mathfrak f:\mathcal K \to
      \mathcal K'$ is a set theoretical function $f:\Pi\rightarrow
      \Pi'$ satisfying the following conditions valid for every $P,P'
      \subseteq \Pi$ (frequently we use the same notation for the
      morphism $\mathfrak f$ and the set theoretical function
      $f$): \begin{itemize} \item[a)]
\label{defapplicativeaksconda}If $\mathrm{QP} \cap {}^\perp P \neq \text{\O}$, then $\mathrm{QP'} \cap {}^\perp f(P) \neq \text{\O}$;
\item[b)]\label{defapplicativeakscondb} \[\mathrm{QP'} \cap
  \bigcap_{\{P,P': \mathrm{QP} \cap {}^\perp(P' \to P) \neq
    \text{\O}\}} {}^\perp\big(f(P'\to P) \to f(P')\to f(P)\big) \neq
  \text{\O}\] \end{itemize} \item \label{defdenseakscond}A
    \textit{computationally dense morphism} is an \emph{applicative
      morphism} $\mathfrak f:\mathcal K \to \mathcal K'$ ($f: \Pi \to
    \Pi' $) with the additional property that there exists a monotonic
    function $h:\mathcal{S}_{\mathcal K'}\rightarrow
    \mathcal{S}_{\mathcal K}$ where $\mathcal{S}_{\mathcal
      K}=\{R\subseteq\Pi:\mathrm{QP} \cap {}^\perp R \neq
    \text{\O}\}$ and $\mathcal{S}_{\mathcal
      K'}=\{R\subseteq\Pi':\mathrm{QP}' \cap {}^\perp R \neq
    \text{\O}\}$ (see paragraph \ref{item:akstoia},\textbf{ From
      abstract Krivine structures to implicative algebras},
    \eqref{item:eseka}) satisfying: \[\mathrm{QP'} \cap \bigcap_{\{R:
      \mathrm{QP}' \cap {}^\perp R \neq
      \text{\O}\}}{}^\perp\big(\mathcal P(f)(h(R)) \to R \big) \neq
    \text{\O}\] \end{enumerate} \end{defi} \begin{obse} We abbreviate
  the above conditions as follows. \begin{enumerate} \item
\begin{itemize} \item[a)] If $P \in \mathcal S_\mathcal K$ then $f(P) \in \mathcal S_{{\mathcal K}'}$; \item[b)] \[\bigcup_{\{P,P': P' \to P
\in \mathcal S_{\mathcal K}\}} \big(f(P'\to P) \to f(P')\to f(P)\big)
  \in \mathcal S_{\mathcal K'};\] \end{itemize} \item There exists a
monotonic function $h:\mathcal{S}_{\mathcal K'} \rightarrow
\mathcal{S}_{\mathcal K}$ satisfying: \[\bigcup_{R \in \mathcal
  S_{\mathcal K'}} \big(\mathcal P(f)(h(R)) \to R \big) \in \mathcal
S_{\mathcal
  K'}.\] \end{enumerate} \end{obse} \begin{lema}\label{lema:cataks}
  Let $\mathcal K \stackrel{}{\rightarrow} \mathcal L
  \stackrel{}{\rightarrow} \mathcal M$ be abstract Krivine structures
  based on the sets $\Pi_1\,,\,\Pi_2$ and $\Pi_3$ respectively and let
  $f:\Pi_1 \to \Pi_2$ and $g:\Pi_2 \to \Pi_3$ be set theoretical
  functions.
\begin{enumerate} \item If $f$ and $g$ correspond to applicative morphisms $\mathfrak f$ and $\mathfrak g$ so does $gf$. We define the
morphism $\mathfrak g \mathfrak f$ as corresponding to $gf$. \item If
$\mathfrak f$ and $\mathfrak g$ are computationally dense morphisms so
is $\mathfrak g \mathfrak f: \mathcal K \to \mathcal M$. \item The map
$\operatorname{id}_\mathcal K:=\operatorname{id}_\Pi:\mathcal K \to
\mathcal K$ is a computationally dense
morphism. \end{enumerate} \end{lema} \begin{proof} The proof follows
    from direct manipulation of the definitions. The first condition
      (1) a) of Definition \ref{defi:moraks} for the composition
      $\mathfrak g \mathfrak f$ is clear from the corresponding
      conditions for the components. Conditions (1) b) follow from the
      fact that it is just the translation of the corresponding
      condition for implicative algebras. Conditions (2) follow by the
      same reasons. \end{proof}

In accordance with the results just proved, the following definition
is adequate.

\begin{defi}\label{defi:ralgcat2} We define the following
  categories. \begin{enumerate} \item ${\bf AKS}$ that has abstract
    Krivine structures as objects and applicative morphisms as arrows.
     \item ${\bf AKSc}$ with
    the same objects and computationally dense morphisms
    as arrows.
\end{enumerate} \end{defi}

\section{Functoriality of the constructions}\label{section:functoriality} \item\label{item:functoriality} In what follows we prove that the
constructions $A$ and $K$ are in fact functors in the corresponding
categories.

Consider $\mathcal K,\mathcal K' \in \mathcal {AKS}$s and the
associated implicative algebras $(A,\leq,\rightarrow, \mathcal
{S}_A):=A(\mathcal K)$ and $(B,\leq,\rightarrow, \mathcal S_{\mathcal
  B}):=A(\mathcal
K')$.\begin{prop}\label{prop:aplmor}\label{prop:compdensemor} If
  $\mathfrak f:\mathcal K\rightarrow \mathcal K'$ is an applicative
  morphism between the AKSes ($f:\Pi \to \Pi'$), then
  $\mathcal{P}(f):\mathcal P(\Pi) \to \mathcal P(\Pi')$ is an
  applicative morphism between the corresponding implicative algebras
  $A(\mathcal K), A(\mathcal K')$. Moreover if $\mathfrak f$ ($f:\Pi
  \to \Pi'$) is computationally dense, so is $\mathcal
  P(f)$. \end{prop} \begin{proof} We need to check that in our
  context, the three conditions of Definition
  \ref{defi:morfimplalgsimpler} hold. \begin{enumerate} \item We start
    proving that $\mathcal{P}(f)(\mathcal {S}_{\mathcal A})\subseteq
    {\mathcal S_{\mathcal B}}$. We know that $P \in {\mathcal
      S}_{\mathcal A}$ if and only if $\exists \, t \bot P, \,t \in
    \mathrm{QP}$ and condition (1) a) of Definition \ref{defi:moraks}
    above, implies that $\exists s \in \mathrm{QP'}$ such that $s \bot
    f(P)$ and then $f(P) \in {\mathcal S}_{\mathcal B}$.
\item The statement of condition (2) is the following, there is an
  element $R \in {\mathcal S}_{\mathcal B}$ such that $\forall P,Q \in
  \mathcal P(\Pi)=A$ such that $P \to Q \in {\mathcal S}_{\mathcal
    A}$, then $R \supseteq f(P \to Q) \to f(P) \to f(Q)$. It is clear
  that from the hypothesis (condition (1) b) of Definition
  \ref{defi:moraks}) we can deduce the existence of an element $r \in
  \mathrm{QP}'$ such that $r \in {}^\perp(f(P \to Q)\to f(P) \to
  f(Q))$. This implies that the set $\{r\}^\perp \in {\mathcal
    S}_{\mathcal B}$ does the job required for $R$ as this relation
  implies that $\{r\}^\perp \supseteq f(P \to Q) \to f(P)\to
  f(Q)$. \item Assume that $X \subseteq \mathcal P(\Pi)$ and take
  $\bigcurlywedge X =\bigcup \{P: P \in X\}$. Then $f(\bigcurlywedge
  X)=f(\bigcup \{P: P \in X\})=\bigcup\{f(P): P \in X\}=
  \bigcurlywedge\{f(P): P \in X\}$. \end{enumerate} For the proof
  related to the computational density of the image morphism
  $A(\mathfrak f)$ we proceed as follows. The fact that $\mathfrak f$
  is computationally dense guarantees the existence of a monotonic map
  $h: \mathcal S_{\mathcal B}\to \mathcal S_{\mathcal A}$ such
  that: \begin{equation}\label{eqn:compdense}\mathrm{QP'} \cap
    \bigcap_{R \in \mathcal S_{\mathcal B}} {}^\perp\big(\mathcal
    P(f)(h(R)) \to R \big) \neq \text{\O}. \end{equation} The
  condition of computational density at the level of the $IA$s is the
  existence of a map similar to the $h$ above except that the required
  condition is that $\forall P \in \mathcal S_{\mathcal B}$ the
  following relation holds: $\mathcal P(f) \circ h\vdash_{\mathcal
    S_{\mathcal B}} \operatorname{id}_{\mathcal S_{\mathcal B}}$. It is clear that
  this relation and equation \eqref{eqn:compdense} are the
  same. \end{proof}

The results of Proposition \ref{prop:aplmor} together with the
previous definitions and constructions guarantee the result
that follows.
\begin{theo} \label{thm:main1}The map $A: \mathcal {AKS} \to \IA$ extends to: \begin{enumerate} \item A functor $A: {\bf AKS} \to {\bf IA}$
when we take applicative morphisms as arrows; \item A functor also
called $A: {\bf AKSc} \to {\bf IAc}$ when we take the respectively
lluf (wide) subcategories with morphisms that are computationally
dense. \end{enumerate} \end{theo} Now, we deal with the construction
$K$ --given at the level of objects in Paragraph \ref{item:main2}--
and prove that it can be extended to a functor in the corresponding
categories. \begin{prop}\label{prop:mainadj} If $f: \mathcal A
  \rightarrow \mathcal B$ is an applicative morphism between
  implicative algebras, then $K(f):=f: K(\mathcal A)\to K(\mathcal B)$
  is an applicative morphism between the corresponding
  AKSes. Moreover, if $f:A \to B$ is computationally dense, so is
  $K(f)$.\end{prop} \begin{proof} Call $A$ and $B$ the corresponding
  underlying sets of $K(\mathcal A)$ and $K(\mathcal B)$
  respectively. As the characterization of the morphisms at the level
  of AKSes is given in terms of conditions in $\mathcal P(\Pi)$, in
  the case of AKSes of the form $K(\mathcal A)$ we have to consider
  the application and implication at the level of $\mathcal P(A)$. To
  avoid notational confusions we indicate the maps at the level of $A$
  as $\to$ and $\cdot$ and at the level of $\mathcal P(A)$ with the
  symbols but with a subindex: $\to_{\mathcal P}$ and $\cdot_{\mathcal
    P}$. \begin{itemize} \item[$\ast$] The first condition in the
    definition of an applicative morphism of AKSes (Definition
    \ref{defi:moraks}) in our context reads as: if there is a $t\leq
    P$ with $t\in \mathcal S_{\mathcal A}$ then there is a $t' \leq
    f(P)$ with $t'\in \mathcal S_{\mathcal B}$. Clearly, the element
    $t'=f(t)$ does the job because $f$ is monotonic and sends
    $\mathcal S_{\mathcal A}$ into $\mathcal S_{\mathcal B}$. For the
    proof of the second condition, we have that
    according to the hypothesis there is an $r\in \mathcal S_{\mathcal
      B}$ such that $r\leq f(a\rightarrow b)\rightarrow
    f(a)\rightarrow f(b)$ for every $a\rightarrow b\in \mathcal
    S_{\mathcal A}$.

Since $K(f)=f$, we need to find an $r\in \mathcal S_{\mathcal B}$ such
that: $$r\leq f(P\rightarrow_{\mathcal P} Q)\rightarrow_{\mathcal P}
f(P)\rightarrow_{\mathcal P} f(Q)$$ for every $P\rightarrow_{\mathcal
  P} Q$ with $\mathcal S_{\mathcal B}\cap
{}^\perp(P\rightarrow_{\mathcal P} Q)\neq\varnothing$ where
$P\rightarrow_{\mathcal P} Q={}^\perp P\cdot Q$. We have that
$f(P\rightarrow_{\mathcal P} Q)\rightarrow_{\mathcal P}
f(P)\rightarrow_{\mathcal P} f(Q)=\{\alpha\rightarrow\beta: \alpha\leq
f(z\rightarrow w), \forall z\leq P, \forall w\in Q,\,
\beta=x\rightarrow y,  x\leq f(P), y\in f(Q)\}$. We
define a subset $X_{P,Q} \subseteq B$ such that for every $u\in
f(P\rightarrow_{\mathcal P} Q)\rightarrow_{\mathcal P}
f(P)\rightarrow_{\mathcal P} f(Q)$ there is an element $v\in X_{P,Q}$
such that $v\leq u$. We finish the proof by showing the existence a
uniform $r\in\mathcal S_{\mathcal B}$ such that $r\leq
X_{P,Q}$. Consider the set given by
$X_{P,Q}=\{f(\operatorname{inf}(P)\rightarrow v)\rightarrow
f(\operatorname{inf}(P))\rightarrow f(v): v\in Q\}$. If
$\alpha\rightarrow\beta\in f(P\rightarrow_{\mathcal P}
Q)\rightarrow_{\mathcal P} f(P)\rightarrow_{\mathcal P} f(Q)$ then
$\beta=x_0\rightarrow y_0$ where $y_0=f(v_0)$ for some $v_0\in
Q$. Since $x_0\leq f(P)$ then $x_0\leq
\operatorname{inf}(f(P))=f(\operatorname{inf}(P))$ ($f$ preserves
$\operatorname{infima}$). On the other hand, $\alpha\leq
f(\operatorname{inf}(P)\rightarrow v_0)$ since
$\operatorname{inf}(P)\leq P$ and $v_0\in Q$. Thus we have $\alpha\leq
f(\operatorname{inf}(P)\rightarrow v_0)$ and
$f(\operatorname{inf}(P))\rightarrow f(v_0)\leq x_0\rightarrow
y_0=\beta$. Therefore, $f(\operatorname{inf}(P)\rightarrow
v_0)\rightarrow f(\operatorname{inf}(P))\rightarrow
f(v_0)\leq\alpha\rightarrow \beta.$ Since $f$ is an applicative
morphism between implicative algebras we have that $r\leq
f(\operatorname{inf}(P)\rightarrow q)\rightarrow
(f(\operatorname{inf}(P))\rightarrow f(q))$ for every $P$ and $ q\in
Q$ such that $\mathcal S_{\mathcal B}\cap
{}^\perp(P\rightarrow_{\mathcal P} Q)\neq\varnothing$. For this,
$\mathcal S_{\mathcal B}\cap {}^\perp(P\rightarrow_{\mathcal P}
Q)\neq\varnothing$ implies that $r\rightarrow q\in\mathcal S_{\mathcal
  A}$ for every $r\leq P$ and $q\in Q$. In particular
$\operatorname{inf}(P)\rightarrow q\in \mathcal S_{\mathcal
  A}$. \item[$\ast$] Next we deal with the density condition proving
that if $f$ is computationally dense, so is $K(f)$. Assume the
hypothesis holds: there is a monotonic function $h:\mathcal
S_{\mathcal B}\rightarrow \mathcal S_{\mathcal A}$ such that
$f \circ h\vdash_{\mathcal S_{\mathcal
    B}} $. We need to find a function
$\widehat{h}:\mathcal{S}_{A(K(\mathcal{B}))} \rightarrow
\mathcal{S}_{A(K(\mathcal{A}))}$ (being
$\mathcal{S}_{A(K(\mathcal{A}))}=\{R\subseteq A: \mathcal{S}_{\mathcal
  A}\cap {}^\perp R \neq \text{\O}\}$ and
$\mathcal{S}_{A(K(\mathcal{B}))}=\{R\subseteq B:\mathcal{S}_{\mathcal
  B} \cap {}^\perp R \neq \text{\O}\}$)
satisfying:\[\mathcal{S}_{\mathcal B} \cap \bigcap_{\{R:
  \mathcal{S}_{\mathcal B} \cap {}^\perp R \neq \text{\O}\}}
{}^\perp\big(f(\widehat{h}(R))\to_{\mathcal P} R \big) \neq
\text{\O}\] where the notion of orthogonality is in
$K(\mathcal{B})$. We define $\widehat{h}(P)=\{h(p): p\in P\}$. Notice
that if $P\in\mathcal{S}_{A(K(\mathcal{B}))}$ then there is a $b\in
\mathcal{S}_{\mathcal B}$: $b\leq P$ which implies that $h(b)\leq
\widehat{h}(P)$ with $h(b)\in \mathcal{S}_{\mathcal A}$ since $h$ is
monotonic, i.e.,
$\widehat{h}(P)\in\mathcal{S}_{A(K(\mathcal{B}))}$. An easy
computation shows that $f(\widehat{h}(R)) \to R=\{q\rightarrow p:
q\leq f(h(R)), p\in R\}$. According to the orthogonality condition, it
will be enough to find an element $a\in\mathcal{S}_{\mathcal A}$ with
$a\leq q\rightarrow p$, $\forall q\leq f(h(R))$, $\forall p\in R$,
$\forall R \in\mathcal{S}_{A(K(\mathcal{B}))}$. We observe that $
f(h(p))\rightarrow p\leq q\rightarrow p$ for every $p\in R$ and also
that (by hypothesis) there exists an $a \in\mathcal{S}_{\mathcal A}$
such that $ a \leq f(h(b))\rightarrow b$ for every $b\in
\mathcal{S}_{\mathcal B}$. As $\mathcal{S}_{\mathcal B} \cap {}^\perp
R \neq \text{\O}$ in our context implies that $R\subseteq
\mathcal{S}_{\mathcal B}$ the element $a$ mentioned just above, does
the required job. \end{itemize} \end{proof} The results of Proposition
\ref{prop:mainadj} together with the previous definitions and
constructions guarantee the result that follows.

\begin{theo} \label{thm:main2}The map $K: \IA \to \mathcal{AKS}$
  extends to:
  \begin{enumerate}
  \item A functor $K: {\bf IA} \to {\bf AKS}$ when we take applicative
    morphisms as arrows; \item A functor also called $K: {\bf IAc} \to
    {\bf AKSc}$ when we take the lluf (wide) subcategories
    respectively and morphisms that are the computationally
    dense. \end{enumerate} \end{theo}

\section{The main adjunction}\label{section:adjunction}
\item\label{basicdiagram} In this section, we prove that the functors
  described above (see Theorems \ref{thm:main1} and \ref{thm:main2})
  form an adjoint pair in any of the two following
  contexts. \begin{equation}\label{eqn:mainadjunction} \quad\xymatrix{
      {\bf AKS} \ar@/^2pc/[rr]^{A} &
      {\hspace*{-0.19cm}\mbox{\larger[0]$\perp$}} &
      \ar@/^2pc/[ll]^{K}{\bf IA}}\quad\quad \xymatrix{ {\bf AKSc}
      \ar@/^2pc/[rr]^{A} & {\hspace*{-0.19cm}\mbox{\larger[0]$\perp$}}
      & \ar@/^2pc/[ll]^{K}{\bf IAc}} \end{equation}
  
Before the proof of the above statement, we identify explicitly the
composition functors $KA:{\bf AKS} \to {\bf AKS}$ and $AK: {\bf IA}
\to {\bf IA}$.
\begin{enumerate}
\item \label{item:AK} If $\mathcal A=(A,\leq,\to,\mathcal S_A)$, then $A(K(\mathcal A))=(\mathcal
P(A),\leq,\leadsto, \mathfrak S)$
where:
\begin{enumerate} \item $C \leq D$ if and only if $C \supseteq
  D$;
\item For $C,D \subseteq A\,,\,C \leadsto D:=\{c \to d:c
  \leq\bigcurlywedge C, d \in D\}$ --notice that $\leadsto$ is
  an abbreviation of the implication $\to_{A(K(\mathcal A))}$ in
  $A(K(\mathcal A))$;
\item Call $\mathfrak k, \mathfrak s$ the infima in $A(K(\mathcal A))$
  that were considered in Definition \ref{defi:mainimpl},(3) and
  similarly for $\ck,\,\cs$ that are defined as minima en $A$. Then $\mathfrak k \subseteq {\uparrow}\,\ck\,$, $\mathfrak s\subseteq {\uparrow}\,\cs\,$;
\item $\mathfrak S=\{C \subseteq A:
  \bigcurlywedge C \in \mathcal
  S_A\}$.
\end{enumerate}
\item \label{item:KA} The description of
$KA$ is more involved and we write all the operations, relations and
maps in $KA(\mathcal K)$ in terms of the corresponding operations in
$\mathcal K$.  To distinguish one from the others the ones in
$AK(\mathcal K)$ will be adorned with a $\sim$. We write: \[\mathcal
K=(\Pi,\bbot,\mathrm{push}, \mathfrak{app},\cK,\cS,\mathrm{QP})\,,\,
K(A(\mathcal K))=(\mathcal
P(\Pi),\bbot_\sim,\widetilde{\mathrm{push}},
\widetilde{\mathfrak{app}},\widetilde{\cK}, \widetilde{\cS},
\widetilde{\mathrm{QP}}),\] and recall that $\mathrm{push}: \Pi \times
\Pi \to \Pi$ is denoted as $\mathrm{push}(s,\pi)=s \cdot \pi$.
We used above the following abbreviations:
\begin{enumerate}
\item
  $\mathcal P(\Pi)$ is the power set of $\Pi$ and is the basic set of
  $KA(\mathcal K)$.
\item we have $\bbot_{\sim}= \supseteq$\,,\, with $\bbot_\sim\subseteq
  \mathcal P(\Pi) \times \mathcal P(\Pi)$, i.e., $P \perp_{\sim} Q
  \Leftrightarrow P \supseteq Q$.
\item $\widetilde{\mathrm{push}}=\tilde{\cdot}:\mathcal P(\Pi) \times
  \mathcal P(\Pi) \to \mathcal P(\Pi)_\Pi$,
  $\widetilde{\mathrm{push}}(P,Q)={}^\perp P \cdot Q = P \to
  Q$. Recall also the notation $\widetilde{\mathrm{push}}(P,Q)=
  P~\tilde{\cdot}~Q$;
\item $\widetilde{\mathfrak{app}}:\mathcal P(\Pi) \times \mathcal
  P(\Pi) \to \mathcal P(\Pi)$, $\widetilde{\mathfrak{app}}(P,Q)=P
  \cdot_\to Q$ where $\cdot_\to$ is the adjoint operation to $\to$;
\item $\widetilde{\mathrm{QP}}=\{P \subseteq \Pi: {}^\perp P \cap
  \mathrm{QP} \neq \text{\O}\}$.
\end{enumerate}
The explicit expressions for the combinators are: $\widetilde{\cK}=\{t
\cdot u \cdot \pi: t \perp \pi\} \subseteq \{\cK\}^\perp;
\widetilde{\cS}=\{t \cdot u \cdot v \cdot \pi: tv(uv) \perp \pi\}
\subseteq \{\cS\}^\perp$.
\end{enumerate}
\begin{obse}\label{obse:interpretationleadsto} In the above notations,
  Lemma \ref{lema:equalinf} can
  be written as: \[\bigcurlywedge(C \leadsto D)=\bigcurlywedge C \to
  \bigcurlywedge D. \]
\end{obse}
\item\label{item:maincompdenseappl} The theorem that follows, that is
  one of the main objectives of this paper, is formulated for
  computationally dense morphisms, but it is also true for morphisms
  that are merely applicative. This follows from the structure of the
  proof where the applicative case is proved first and without using
  the density conditions.

  \begin{theo}\label{theo:mainadj} The
    functors\, ${\bf AKSc}\stackrel{A}{\longrightarrow}{\bf IAc}$\,
    and\, ${\bf IAc}\stackrel{K}{\longrightarrow}{\bf AKSc}$\, form an
    adjoint pair with counit $\varepsilon:AK \Rightarrow
    \operatorname{id}_{\bf IAc}: {\bf IAc} \rightarrow {\bf IAc}$ and
    unit $\eta :\operatorname{id}_{\bf AKSc} \Rightarrow KA: {\bf
      AKSc} \rightarrow {\bf AKSc}$ defined as
    follows: \begin{enumerate} \item For $\mathcal A \in {\bf IAc}$
      the natural transformation $\varepsilon:AK \Rightarrow
      \operatorname{id}$ is defined on the component $\mathcal A$ as
      the map $\varepsilon_{\mathcal A}:\mathcal P(A) \to A$ with
      $\varepsilon_{\mathcal A}(C)=\bigcurlywedge C$ for $C \subseteq
      A$. \item For $\mathcal K \in {\bf AKSc}$ the natural
      transformation $\eta:\operatorname{id} \Rightarrow KA$ on the
      component $\mathcal K$ is defined as the map $\eta_{\mathcal
        K}:\Pi \to\mathcal P(\Pi)$ with $\eta_{\mathcal
        K}(\pi)=\{\pi\}$ for $\pi \in
      \Pi$.
    \end{enumerate}
  \end{theo}
  \begin{proof}
    \begin{enumerate}
    \item
      The counit.
\begin{enumerate} \item We start proving that $\varepsilon_{\mathcal A}:\mathcal P(A) \to A$: $\varepsilon_{\mathcal A}(C)=\bigcurlywedge C$
for $C \subseteq A$ is a morphism in the category ${\bf IA}$. The
proof that the map $\varepsilon_{\mathcal A}$ takes the separator
$\mathfrak S$ into $\mathcal S_{\mathcal A}$ follows immediately from
the description of $\mathfrak S$ given above in paragraph
\ref{basicdiagram} (item \eqref{item:AK} (c)). Next, we show that
$\varepsilon_{\mathcal A}$ satisfies the condition of Definition
\ref{defi:morfimplalgsimpler} (2). We need to prove that there exists
an $r \in \mathcal S_{\mathcal A}$ such that for all $C,D \subseteq A$
such that $C \leadsto D \in \mathfrak S$, $r \leq
\varepsilon_{\mathcal A}(C \leadsto D)\to (\varepsilon_{\mathcal A}(C)
\to \varepsilon_{\mathcal A}(D))=\bigcurlywedge_A(C \leadsto D)\to
(\bigcurlywedge_A(C) \to \bigcurlywedge_A(D))=\bigcurlywedge\{c \to d:
c \leq \bigcurlywedge C, d \in D\}\to (\bigcurlywedge_A(C) \to
\bigcurlywedge_A(D))$. As both terms of this expression are equal, we
can take $r=\ci$ (see Lemma \ref{lema:equalinf}). We finish this part
by proving that $\varepsilon_{\mathcal A}$ commutes with infinite
meets. In our case the order in the domain is the reverse inclusion,
hence the required commutation becomes the following equality for an
arbitrary family of subsets $\{C_i: i \in I\} \subseteq\mathcal P(A)$:
$\bigcurlywedge_A(\bigcup_i
C_i)=\bigcurlywedge_A\{\bigcurlywedge_A(C_i): i \in I\}$.  This
equality is clearly true. \item The fact that $\varepsilon: \mathcal
P(A) \to A$ is computationally dense, follows from the existence of a
right inverse $h: a \mapsto\,{\uparrow} a: A \to\mathcal P(A)$.
Observe also that $h$ is monotonic and sends the separator of $A$ into
$\mathfrak S=\varepsilon_{\mathcal A}^{-1}(\mathcal S_{\mathcal A})$.
\item Next we prove that the family of morphisms
  $\varepsilon_{\mathcal A}$ is natural, i.e. that for any applicative
  morphism $f: \mathcal A \to \mathcal B$ of implicative algebras, the
  following diagram commutes: \begin{center} \[\xymatrix{\mathcal
      P(A)\ar[d]_{\mathcal P(f)}\ar[rr]^{\varepsilon_{\mathcal A}}&&
      A\ar[d]^f\\\mathcal P(B)\ar[rr]^{\varepsilon_{\mathcal A}}&&
      A}\] \end{center} The commutativity of this diagram is the last
  condition of the definition of morphism in ${\bf IA}$ and follows
  directly from the fact that $f$ preserves
  $\bigcurlywedge$. \end{enumerate} \item The
unit. \begin{enumerate} \item We start by proving that the map
  $\eta_{\mathcal K}:\Pi \to \mathcal P(\Pi)$,\, $\eta_{\mathcal
    K}(\pi)=\{\pi\}$ is a morphism in ${\bf AKS}$, i.e. is applicative
  in the sense of Definition \ref{defi:moraks}. We check conditions
  a), b) therein. Condition a) is an easy consequence of the
  following: the implication: ($\mathrm{QP} \cap {}^\perp P \neq
  \text{\O}\Rightarrow \mathrm{QP'} \cap {}^\perp \eta_{\mathcal
    K}(P)\neq \text{\O}$) translates into if $\exists s\in
  \mathrm{QP}$ such that $s \perp P$ then $\exists S:{}^\perp S \cap
  \rmqp\neq \text{\O}$ with $S \perp_{\sim} \eta_{\mathcal K}(P)$ (
  i.e., $S \perp \eta_{\mathcal K}(p)$ for every $p\in P$). Hence, if
  we take $S=P$ it is clear that $s \in {}^\perp P \cap \rmqp$ and
  hence the intersection is not empty. Moreover, since $\forall p\in
  P,\, \eta_{\mathcal K}(p)=\{p\}\subseteq P$ and this means $\forall
  p\in P,\, P\perp_{\sim} \eta_{\mathcal K}(p)$.

Condition b) can be formulated as follows: there is a set $S \subseteq
\Pi$ such that ${}^\perp S \cap \rmqp \neq \text{\O}$ and $S
\perp_{\sim} \eta_{\mathcal K}(P \to Q) \to\eta_{\mathcal K}(P) \to
\eta_{\mathcal K}(Q)$ for all $P,Q\subseteq \Pi$ such that there
exists some $t_{P,Q} \in\rmqp$ and $t_{P,Q} \perp {}^\perp P\cdot
Q$. Observe that the above inequality for $S$ is equivalent with $S
\perp_{\sim} {}^\perp\big(\eta_{\mathcal K}({}^\perp P \cdot
Q)\big)\,\tilde{\cdot}\, \big(\eta_{\mathcal K}(P) \to\eta_{\mathcal
  K}(Q)\big)$. We start by finding an upper bound for the rightmost
term of the above relation. A subset of $\mathcal P(\Pi)$ of the form
${}^{\perp_\sim}\eta_{\mathcal K}(M)$ for $M \subseteq \Pi$ equals:
$\{R \subseteq \Pi: R \supseteq \{\pi\}, \forall \pi \in M\}=\{R
\subseteq \Pi:R \supseteq M\}$. Hence: \begin{itemize} \item
  ${}^{\perp_\sim}\eta_{\mathcal K}({}^\perp P\cdot Q)=\{R \subseteq
  \Pi: R \supseteq {}^\perp P \cdot Q\}$; \item $\eta_{\mathcal K}(P)
  \to \eta_{\mathcal K}(Q)={}^{\perp_\sim}(\eta_{\mathcal K}(P))
  \,\tilde{\cdot}\, \eta_{\mathcal K}(Q)=\{M \,\tilde{\cdot}\,\{\pi\}:
  M \supseteq P\,\, \pi \in Q\}=\{{}^\perp M \cdot \{\pi\}: M
  \supseteq P\,,\, \pi \in Q\}$;
\item ${}^{\perp_\sim}\big(\eta_{\mathcal K}({}^\perp P \cdot
  Q)\big)\,\tilde{\cdot}\, \big(\eta_{\mathcal K}(P) \to\eta_{\mathcal
    K}(Q)\big)=\{R \to {}^\perp M \cdot\{\pi\}: R \supseteq {}^\perp
  P\cdot Q\,,\, M \supseteq P\,,\, \pi \in Q\}$.
\end{itemize}
Hence,
in accordance with the variance of the different maps($\to$ is
contravariant in the first variable and covariant in the second and
$\cdot$ is covariant in both variables) we have the following
bound: \[\bigcup {}^{\perp_\sim}\big(\eta_{\mathcal K}({}^\perp P \cdot
Q)\big)\,\tilde{\cdot}\, \big(\eta_{\mathcal K}(P) \to\eta_{\mathcal
  K}(Q)\big)\subseteq {}^\perp P \cdot Q \to{}^\perp P \cdot Q.\] As
the set $I:=\{\cS\cK\cK\}^\perp\in \widetilde{QP}$ satisfies the
following inequality: ${}^\perp P \cdot Q \to {}^\perp P \cdot Q
\subseteq I$ we can take the $S=I$, and
the proof of condition b) is finished. \item We verify now, that the
maps $\eta_{\mathcal K}: \mathcal K \to K(A(\mathcal K))$ are
computationally dense morphisms. We consider the function
$h:\mathcal{S}_{A(K(A(\mathcal K)))}
\rightarrow\mathcal{S}_{A(\mathcal K)}$ defined as follows. An element
of the domain is a family of subsets of $\Pi$, $\mathcal H=\{H_i: i
\in I, H_i \subseteq \Pi\}$ and the function $h$ is defined as the
union: $h(\mathcal H)=\bigcup_i H_i$. Clearly the map $h$ is
monotonic. Moreover, $h(\mathcal{S}_{A(K(A(\mathcal K)))})
\subseteq \mathcal{S}_{A(\mathcal K)}$. Indeed, if $\mathcal H=\{H_i: i
\in I, H_i \subseteq \Pi\} \in \mathcal{S}_{A(K(A(\mathcal K)))}$ there is a $P \subseteq \Pi$ such that ${}^\perp P \cap \operatorname{QP} \neq \text{\O}$ and $P \supseteq H_i$ for all $i \in I$ (recall the explicit description of $KA$ in Paragraph \ref{basicdiagram}). Then $P \supseteq h(\mathcal H)$ and we deduce that $h(\mathcal H) \in \mathcal S_{A(\mathcal K)}$.   To finish the proof we need to prove the existence of a
subset $T \in \mathcal{S}_{A(\mathcal K)}$ satisfying:
$$T \perp_{\sim} \mathcal P(\eta_{\mathcal K})(h(\mathcal H))\rightarrow\mathcal H$$
for every $\mathcal H\in \mathcal{S}_{A(K(A(\mathcal K)))}$.

Thus, since $\mathcal P(\eta_{\mathcal K})(h(\mathcal H))\rightarrow
\mathcal H={}^{\perp_\sim}(\mathcal P(\eta_{\mathcal K})(h(\mathcal
H))) \cdot_\sim \mathcal H=\{Q: Q\perp_{\sim}\{x\}, \forall x\in
\bigcup_i H_i\} \cdot_{\sim} \mathcal H=\{Q\rightarrow H_j: Q\supseteq
\bigcup_i H_i: j \in I\}$. Therefore, we need to be able to find $T$
such that: $$T\supseteq \bigcup_i H_i \rightarrow H_j$$ for every $j
\in I$ and for every $\mathcal H=\{H_i: i \in I\} \in
\mathcal{S}_{A(K(A(\mathcal K)))}$. But, if we
call $I=\{\cS\cK\cK\}^\perp$, then $I\supseteq H_j \rightarrow H_j
\supseteq \bigcup_i H_i \rightarrow H_j$ does the job.

\item The naturality of $\eta_{\mathcal K}$ means the evident fact
  that for all $\mathfrak f:\mathcal K \to \mathcal K'$ (i.e. $f: \Pi
  \to \Pi'$) the diagram written below is
  commutative: \begin{center} \[\xymatrix{\Pi
      \ar[d]_{f}\ar[rr]^{\eta_{\mathcal K}}&& \mathcal
      P(\Pi)\ar[d]^{\mathcal P(f)}\\ \Pi'\ar[rr]^{\eta_{\mathcal
          K'}}&& \mathcal
      P(\Pi')}\] \end{center} \end{enumerate} \item The triangular
equalities. These equalities mean that the following diagrams of
functors and natural transformations are commutative:
\begin{enumerate} \item \begin{equation*} \xymatrix{& AKA \ar[dr]^{\varepsilon A}&\\A \ar[ur]^{A\eta}\ar[rr]_{\operatorname{id}_A} && A}
\end{equation*} \item \begin{equation*} \xymatrix{& KAK \ar[dr]^{K\varepsilon}&\\K \ar[ur]^{\eta K}\ar[rr]_{\operatorname{id}_K} && K}
\end{equation*} \end{enumerate} \begin{enumerate} \item Concerning the first diagram we have that for $\mathcal K \in {\bf AKS}$ the map:
\[A(\eta_{\mathcal K}):A(\mathcal K) \to A(K(A(\mathcal K)))\] is the set-theoretical map from $P \mapsto \{\{\pi\}: \pi \in P\}: \mathcal
P(\Pi) \to \mathcal P(\mathcal P(\Pi))$. Morever, the
map \[\varepsilon_{A(\mathcal K)}: A(K(A(\mathcal K))) \to A(\mathcal
K)\] is the set-theoretical from $\mathscr P \mapsto
\bigcurlywedge_{A(\mathcal K)}=\bigcup \{P : P \in \mathscr P\}:
\mathcal P(\mathcal P(\Pi)) \to \mathcal P(\Pi)$. It is clear that the
composition of both maps is the identity.\item Consider now the
map \[\eta_{K(\mathcal A)}: K(\mathcal A) \to K(A(K(\mathcal A)))\]
that is the set-theoretical map
$a \mapsto \{a\}:A \to \mathcal P(A)$ and $K(\varepsilon_{\mathcal
  A}): K(A(K(\mathcal A))) \to K(\mathcal A)$ that is the set
theoretical map $P \mapsto \bigcurlywedge_A(P): \mathcal P(A) \to
A$. It is clear that the composition of both maps is the identity.
\end{enumerate}
    \end{enumerate}
  \end{proof}
      
  \begin{obse} We recall the following basic facts concerning adjunctions. Assume that ${\mathbf A}$,\,${\mathbf B}$ are arbitrary categories and that the pair $(L,R)$ is an adjunction $\xymatrix@!=.5pc{{\mathbf B}\ar@/^1pc/[rr]^{L}&  
{\hspace*{-0.05cm}\mbox{\larger[1]$\,\perp$}}& {\mathbf
  A}\ar@/^1pc/[ll]^{R}}$.
    \begin{itemize}
    \item[*] The functor $L$ is faithful if and only if for every $B\in
      \mathcal B $ the unit $\eta_B: B \rightarrow RL(B)$ is monic
      . Moreover, $L$ is full if and only if each $B$ the unit
      morphism $\eta_B:B\rightarrow RL(B)$ has a right inverse.
    \item[*] Dually, $R$ is faithful if and only if for all $A\in
      \mathcal A$, $\varepsilon_A:LR(A)\rightarrow A$ is epi . Also,
      $R$ is full if and only if each $\varepsilon_A:LR(A)\rightarrow
      A$ has a left inverse. 
    \end{itemize}
    See \cite[Chap. IV, Sect. 3]{kn:cwm}.
  \end{obse}
  \begin{coro}\label{coro:adjplus} In the notations of
    Theorem \ref{theo:mainadj} the functors ${\bf
      AKSc}\stackrel{A}{\longrightarrow}{\bf IAc}$\, and\, ${\bf
      IAc}\stackrel{K}{\longrightarrow}{\bf AKSc}$ ($A \dashv K$) are
    both faithful but not full, and the same happens with their
    restrictions to ${\bf AKS}$ and ${\bf IA}$.
  \end{coro}
 
  \begin{obse} The two categories that we are considering
    (abstract Krivine structures and implicative algebras), even
    though the above theorem shows that they are quite disparate,
    produce the same class of triposes up to equivalence when when
    mapped by ${\bf H}$ --see Section \ref{section:IAtoTripos}--.
    \end{obse}
  \section{The category of implicative comonads}\label{section:changing} \item In \cite{kn:ocar2}, the authors together with
  M. Guillermo, presented a modification of the construction by
  Streicher of a tripos from an AKS that appeared in
  \cite{kn:streicher}. This new process of producing a tripos was
  named ``the bullet construction''\footnote{The bullet construction
    appeared in \cite{kn:ocar2}, and consists of the consideration of
    an implicative algebra based upon $\mathcal P_\bullet(\Pi)$ that
    is the set of all subsets of $\Pi$ that are closed under the
    Alexandroff approximation of Streicher's closure operator.} and it
  was proved that the resulting tripos was equivalent to Streicher's
  as well as to Krivine's triposes. In the same manner, as the work of
  Streicher was based in the use of the double perpendicular closure
  operator associated with the pole of the original AKS, the mentioned
  modification involved an \emph{Alexandroff operator}, which was
  closely univocally determined by, the double perpendicular operator
  considered by Streicher. Below we present a generalization of the
  ``bullet construction'', in the context of the categories of
  \emph{comonads} in ${\bf IA}$ and ${\bf IAc}$ that is valid for
  Alexandroff operators (comonads) that satisfy some mild additional
  hypotheses that hold true in the case of implicative algebras coming
  from AKSes. We mentioned briefly in the Introduction, the
  relationship of these constructions with the \emph{formal theory of
    monads} as appears in \cite{kn:street}.
\item \label{item:generalalexandroff} We start by
recalling the basic definitions of Alexandroff interior operators and
proving an ``approximation'' result generalizing some of the
constructions in \cite[Section 4]{kn:ocar2} that were established
therein for the situation of implicative algebras coming from
AKSes. \begin{defi}\label{defi:appendix} Let $\mathcal A=(A,\leq)$ be
  an $\operatorname{inf}$-complete
  semilattice. \begin{enumerate} \item An \emph{interior operator} on
    $\mathcal A$ is a map $\iota: a \mapsto \iota(a): A \rightarrow A$
    such that: \begin{enumerate*} \item For all $a \in A$, $\iota(a)
      \leq a$, \item For all $a \in A$, $\iota(\iota(a))=
      \iota(a)$, \item For all $a,b\in A$, $a \leq b$ implies that
      $\iota(a) \leq \iota(b)$. \end{enumerate*} \item It is said to
    satisfy the \emph{Alexandroff condition} if for every $B \subseteq
    A$ we have $\iota(\bigcurlywedge_{b \in B}b)=\bigcurlywedge_{b \in
      B}\iota(b)$. \item The set of $\iota$--open elements of $A$ is
    $A_{\iota}:=\iota(A)=\{a \in A:\iota(a)=a\}.$ \item Call $\mathcal
    I(\mathcal A)$ the set of the interior operators and $\mathcal
    {I_{\!\!\infty}}(\mathcal A)$ the set of Alexandroff interior
    operators of $\mathcal A$.
\item \label{item:order} For $\iota\,,\,\kappa$ interior operators
  write $\iota \leq \kappa$ if for all $a \in A$, $\iota(a)\leq
  \kappa(a)$.
\end{enumerate} \end{defi} 
\begin{obse} It is clear in the above context, that an interior operator in $(A, \leq)$ is the same than a  comonad structure, and that an Alexandroff interior operator is the same than a 
continuous comonad structure.
\end{obse}
\begin{theo}\label{theo:existencealaprox} If $\mathcal{A}=(A, \leq)$ is an $\operatorname{inf}$-complete
semilattice, the inclusion of posets viewed as a functor
$\mathrm{inc}: \mathcal I_{\!\!\infty}(\mathcal A) \subseteq \mathcal
I(\mathcal A)$ has a left adjoint. In other words for all $\iota \in
\mathcal I(A)$ there exists an Alexandroff interior operator
$\iota_{\infty} \in \mathcal I_{\!\!\infty}(\mathcal A)$ such that
$\iota_\infty=\bigcurlywedge\{\tau \in \mathcal
I_{\!\!\infty}(\mathcal A): \iota \leq \tau\}$. In that situation we
say that $\iota_{\infty}$ is the AL-approximation of $\iota$
(Alexandroff approximation). \end{theo}
\begin{proof} For a complete semilattice, we define the following two subsets of $\mathcal P(A)$: $\mathcal P_c(A)=\{B \subseteq A: \forall
C \subseteq B, \sup_A(C) \in B\}$\,,\, $\mathcal P_{c,\infty}(A)=\{B
\in \mathcal P_c(A):\forall C \subseteq B, \bigcurlywedge_A(C) \in B\}
\subseteq \mathcal P_c(A)$ and take in both sets the order given by
inclusion. Then, the maps $\Theta(\iota)=A_\iota=\{x \in
A:\iota(x)=x\}$ and $\Theta^{-1}(B)=\iota_B$ where $\iota_B:A \to A$
is defined as $\iota_B(x)=\sup_A B_x$ where $B_x=B \cap (-\infty,x]=\{b
  \in B: b \leq x\}$; define an ordered bijection between $\mathcal
  I(A)$ and $\mathcal P_{c,\infty}(A)$ such that $\Theta(\mathcal
  I_{\!\!\infty}(\mathcal A))= \mathcal P_{c,\infty}(A)$ --we endow
  $\mathcal I(\mathcal A)$ with the pointwise order and $\mathcal
  P_c(A)$ with the order given by inclusion. Using this bijection we
  change the assertion about $\mathrm{inc}: \mathcal
  I_{\!\!\infty}(\mathcal A) \subseteq \mathcal I(\mathcal A)$ into an
  assertion about $\mathrm{inc}':\mathcal P_{c,\infty}(A) \subseteq
  \mathcal P_{c}(A)$ and we prove that given an inclusion $B \subseteq
  A$ with $B$ closed under sups (i.e. $B \in \mathcal P_{c}(A)$),
  there is a smallest $B_{\infty} \in \mathcal P_{c,\infty}(A)$ such
  that $B \subseteq B_\infty$. Clearly, the subposet
  $\bigcap\{B':B\subseteq B' \subseteq A: B' \in P_{c,\infty}(A)\}$
  does the job\footnote{We thank Ignacio Lopez Franco for suggesting
    the compact argument presented.}. \end{proof}
\begin{obse}\label{obse:closureop}\begin{enumerate} \item The interior operators considered in \cite[Definitions 4.3 and 4.6]{kn:ocar2} are
the double perpendicular closure and its Alexandroff approximation,
that therein are viewed as \emph{closure operators} because  --as it is
frequent in this theory-- the order taken in the power set of
the stacks, is the inverse inclusion order. It is clear that the closure
corresponds to the interior in the opposite order. The closed elements
defined by a \emph{closure operator} $\operatorname{c}$ are the
elements $a$ such that
$\operatorname{c}(a)=a$. \item \label{item:closurint} If $\iota$ is an
\emph{Alexandroff} interior operator, then we define a closure
operator $\operatorname{c}_\iota: A \rightarrow A$ as follows:
$\operatorname{c}_\iota(a)=\bigcurlywedge\{b: \iota(b)=b, a \leq
b\}$. In this context, the set of closed elements of
$\operatorname{c}_\iota$ coincides with the set of open elements of
$\iota$.
  \end{enumerate}
\end{obse}
\begin{defi} Let $\mathcal A=(A,\leq, \to)$ be an implicative structure
  and $\iota:(A,\leq) \to (A,\leq)$ an Alexandroff interior
  operator. Let $A_\iota$ be the set of open elements of $A$, i.e. the
  fixed points of $\iota$. Define the map $\to_\iota: (a,b) \mapsto
  \iota(a \to b):A_\iota \times A_\iota \rightarrow
  A_\iota$ \end{defi}
\begin{theo}\label{theo:changeimpl}\label{lema:change}
  If $\mathcal A=(A,\leq, \to)$ is an implicative structure and
  $\iota$ an Alexandroff interior operator and $c_\iota$ its
  associated closure operator, then:
  \begin{enumerate}
  \item
    $(A_\iota,\leq)$ is an \emph{inf}-complete semilattice;
  \item The map $\to_\iota: A_\iota \times A_\iota \rightarrow
    A_\iota$ defines an implicative structure in $A_\iota$;
\item If $\cdot:A \times A \rightarrow A$ and $\cdot_{\iota}:A_\iota
  \times A_\iota \to A_\iota$ are the application maps associated to
  $\to$ and $\to_\iota$ respectively, then $a\cdot_{\iota}b= c_\iota
  (a\cdot b)$ for every $a$ and $b\in A_\iota\, $ (see diagram
  \ref{eqn:iotace} below).
\end{enumerate} 
\end{theo} \begin{proof}
\begin{enumerate} \item Given $B\subseteq A_\iota$ it is clear that $\bigcurlywedge B \in A_\iota$: $\iota( \bigcurlywedge B)=\bigcurlywedge
\iota(B)=\bigcurlywedge B$. Then the conclusion follows by defining
$\bigcurlywedge_\iota B=\bigcurlywedge B$.
\item Assume that $a \in A_\iota, B \subseteq A_\iota$, then $a
  \to_\iota \bigcurlywedge_\iota B=\iota(a \to \bigcurlywedge
  B)=\iota\left(\bigcurlywedge_{b \in B}(a \to
  b)\right)=\bigcurlywedge_{b \in B}\iota(a \to b)=
  (\bigcurlywedge_\iota)_{b \in B}(a \to_\iota b)$. Monotonicity and
  antimonotonicity properties are immediate.
\item By definition, for $a,b \in A_\iota$, $a \cdot_{\iota}b=
  {\bigcurlywedge}_\iota \{d \in A_\iota: a \leq b \to_\iota
  d\}$. Then, from the following chain of equalities we deduce our
  conclusion. \begin{gather*} \bigcurlywedge_\iota \{d \in A_\iota: a
    \leq b \to_\iota d\} =\bigcurlywedge_\iota \{d \in A_\iota: a \leq
    \iota(b \to d)\}= \bigcurlywedge_\iota \{d \in A_\iota: a \leq b
    \to d\}=\\\bigcurlywedge_\iota \{d \in A_\iota: a\cdot b \leq
    d\}=\operatorname{c}_\iota(a\cdot b). \end{gather*} For the last
  equality see Observation \ref{obse:closureop},
  \eqref{item:closurint}.
\end{enumerate} \end{proof} The commutative diagrams that follows illustrates the relationship between the old and new application and
implication maps of the implicative
structures: \begin{equation}\label{eqn:iotace}\xymatrix{A_\iota \times
    A_\iota \ar[d]_{\operatorname{inc}\times
      \operatorname{inc}}\ar[rr]^{\to_{\iota}}&& A_\iota
    \ar[d]^{\operatorname{inc}}\\ A\times A
    \ar[r]_-{\to}&A\ar[r]_{\iota}&A} \quad \xymatrix{A_\iota \times
    A_\iota \ar[d]_{\operatorname{inc}\times
      \operatorname{inc}}\ar[rr]^{\cdot_{\iota}}&& A_\iota
    \ar[d]^{\operatorname{inc}}\\ A\times A
    \ar[r]_{\cdot}&A\ar[r]_{c_\iota}&A}\end{equation} where
$\operatorname{inc}:A_\iota\to A$ is the inclusion map.
\item \label{item:compinterior} Until now we considered Alexandroff
  interior operators in the context of complete posets or implicative
  structures. In this paragraph --and for future use-- we introduce
  the class of Alexandroff operators that we will be using for
  implicative algebras. They are required to have natural
  compatibility conditions with the implication and the separator.

\begin{defi}\label{defi:alexandroffimpl} Let $\mathcal{A}=(A,\leq,\rightarrow,\mathcal{S})$ be an implicative algebra. An
Alexandroff interior operator $\iota:A \rightarrow A$ in $A$ is said
to be {\em compatible} with $\mathcal A$ --or merely {\em
  compatible}-- if it satisfies condition (1), and {\em strongly
  compatible --with $\mathcal A$--} if it satisfies condition (2):
\begin{enumerate}\item $\forall a \in A, \ci a \leq
\iota(a)$; \item For all $a,b \in A$, $\iota(a) \to b = a \to b$.
\end{enumerate}
\end{defi} 

\begin{obse}\label{obse:alexandroffimpl}\begin{enumerate}
  \item Clearly, the condition (1) in terms of the implication morphism
    can be expressed as $\forall a \in A\,,\, \ci \leq a \to \iota(a)$.

  \item Observe that condition (2) in Definition
\ref{defi:alexandroffimpl} implies condition (1). Indeed, if $a \in A$
we have that $\ci \leq \iota(a) \to \iota(a)= a \to \iota(a)$.
\item Moreover, the condition (1) and hence the stronger condition (2)
  implies that the separator is invariant by $\iota$, i.e. that
  $\iota(\mathcal S) \subseteq \mathcal S$.  Indeed: if we apply
  condition (1) for the situation that $a=s \in \mathcal S$, we have
  that $\ci s \leq \iota(s)$ and as $\mathcal S$ is closed under
  application and upper closed by inequalities, we deduce that
  $\iota(s) \in \mathcal S$.
\item
  Also, it is easy to show that $\iota(\mathcal S) \subseteq \mathcal
S$ if and only if $\iota(\mathcal S)=\mathcal S\cap A_\iota$. If
$\iota(\mathcal S) \subseteq \mathcal S$, it is clear that also
$\iota(\mathcal S) \subseteq \mathcal S \cap A_\iota$.  As the reverse
inclusion $\iota(\mathcal S) \supseteq \mathcal S \cap A_\iota$ is
always true, half of the above assertion is proved. The rest of the proof also follows easily.
\end{enumerate} 
\end{obse}
 
\begin{theo}\label{ejemplosdecompsense} Let $\mathcal A$ be an implicative
algebra with separator $\mathcal S$ and assume that $\iota:A \to A$ is
a compatible Alexandroff interior operator, then the associated map
$\iota:\mathcal A \to \mathcal A$ is a computationally dense
morphism. \end{theo} \begin{proof} In accordance with Definition
  \ref{defi:morfimplalgsimpler} and Lemma
  \ref{lema:morfimplalgsimpler} and the corresponding Definition
  \ref{defi:compdensemor} and Observation \ref{obse:compdensemor},we
  have to prove the following facts: \begin{enumerate} \item
    $\iota(\mathcal S)\subseteq \mathcal S$;
\item $\exists t \in \mathcal S: \forall s \in \mathcal S,
  t\iota(s)\iota(a) \leq \iota(sa)$; \item $\iota(\bigcurlywedge P)=
  \bigcurlywedge\iota(P)$ for all $P \subseteq A$; \item $\exists t
  \in \mathcal S, h:\mathcal S \to \mathcal S: t(\iota(h(s)))\leq s$
  for all $s \in \mathcal S$ and with $h$ monotonic. \end{enumerate}
  Conditions (1) and (3) follow directly from Definitions
  \ref{defi:appendix} and \ref{defi:alexandroffimpl} together with
  Observation \ref{obse:alexandroffimpl}. Regarding condition (2) we
  proceed as follows: it is clear that $\iota(s)\iota(a) \leq sa$ and
  then, $\ci (\iota (s)\iota(a))\leq \ci(sa)\leq \iota(sa)$. Hence,
  using Lemma \ref{lema:asocas} we deduce that if we call $t=\nu \ci$
  the required condition holds. For (4) we take $h=\operatorname{id}$
  and $t=\ci$. It is clear that $\ci\iota(s)\leq \iota(s)\leq
  s$. \end{proof} 
  \item The category of continuous comonads of ${\bf
  IA}$ (or ${\bf IAc}$) named as ${\bf Co(IA)}$ (or ${\bf Co(IAc)}$)
is given by the following data: \begin{enumerate} \item The objects
  are pairs $(\mathcal A,\iota)$ where
  $\mathcal{A}=(A,\leq,\rightarrow,\mathcal{S})$ is an implicative
  algebra and $\iota:\mathcal A\rightarrow \mathcal A$ is an
  applicative (computationally dense) morphism that is also a
  comonad.  \item Given two objects $(\mathcal A,\iota)$, $
  (\mathcal B,\kappa)$ an arrow $f:(\mathcal A,\iota)\rightarrow
  (\mathcal B,\kappa)$ is defined by an arrow $f:A\rightarrow B$ in
  the category ${\bf IA}$ (or ${\bf IAc}$) such that the diagram below
  commutes: \[\xymatrix{A
    \ar[d]_{\iota}\ar[rr]^{f}&&B\ar[d]^{\kappa}\\A\ar[rr]_{f}&&B.}\]
    \end{enumerate}
    
     Next, we introduce
the {\em category of continuous comonads} along ${\bf IA}$ (or ${\bf
  IAc}$) that is a full subcategory of ${\bf Co(IA)}$ (or ${\bf
  Co(IAc)}$) that we name as the category of implicative comonads and
we denote it as ${\bf Co_{\operatorname{imp}}(IA)}$(or ${\bf
  Co_{\operatorname{imp}}(IAc)}$). Our construction requires that the
objects $(\mathcal A,\iota)$ of this category satisfy the additional
conditions that $\iota$ is compatible with $\mathcal A$, see
Definition \ref{defi:alexandroffimpl}. See for example
\cite{kn:street} for the general theory of (co)monads 
that underlies our concrete constructions above. 
\begin{defi} Given
  the category ${\bf IA}$ (or ${\bf IAc}$) of implicative algebras
  with applicative (computationally dense) morphisms, we define a
  category ${\bf Co_{\operatorname{imp}}(IA)}$ (or ${\bf
    Co_{\operatorname{imp}}(IAc)}$) determined by the following
  data: \begin{enumerate} \item The objects are pairs $(\mathcal
    A,\iota)$ where $\mathcal{A}=(A,\leq,\rightarrow,\mathcal{S})$ is
    an implicative algebra and $\iota:A\rightarrow A$ is a compatible
    Alexandroff interior operator; \item Given two objects $(\mathcal
    A,\iota)$, $ (\mathcal B,\kappa)$ an arrow $f:(\mathcal
    A,\iota)\rightarrow (\mathcal B,\kappa)$ is defined by an arrow
    $f:A\rightarrow B$ in the category ${\bf IA}$ (or ${\bf IAc}$)
    such that the diagram below commutes:
\[\xymatrix{A \ar[d]_{\iota}\ar[rr]^{f}&&B\ar[d]^{\kappa}\\A\ar[rr]_{f}&&B.}\]
 \end{enumerate} \end{defi} 
The following theorems complete
the results of Theorem \ref{theo:changeimpl} and present it as
defining a functor from ${\bf Co_{imp}(IA)}$ into ${\bf IA}$ (${\bf
  Co_{imp}(IAc)}$ into ${\bf IAc}$). We need to perform some
computations before defining the needed functor (the
\emph{construction of algebra functor} in \cite{kn:street}). Our
aspiration now is to generalize to the current general platform the
constructions appearing in \cite[Def. 4.7, 4.15]{kn:ocar2} of
operations in the set $\mathcal P(\Pi)$ for $\Pi$ the set of stacks of
an AKS.
\begin{theo} \label{theo:implicativeiota} Let $\mathcal A$ be an implicative algebra, $\iota:A \to A$ a compatible Alexandroff interior operator,
$(A_{\iota},\cdot_\iota, \to_{\iota})$ as in Theorem
  \ref{theo:changeimpl}, call $\mathcal S_\iota=\iota(\mathcal
  S)=\mathcal S \cap A_\iota$ and define the elements
  $\ck_{\iota}:=\iota(\nu \ci \ck), \cs_{\iota}:=\iota(\nu (\nu \ci
  (\nu \ci)) \cs) \in \mathcal S \cap
  A_\iota$. Then: \begin{enumerate} \item For all $a,b \in A_\iota$,
    $\ck_\iota \leq a\to_\iota(b \to_\iota a)$; \item For all $a,b,c
    \in A_{\iota}$ $\cs_{\iota} \leq a\rightarrow_\iota
    (b\rightarrow_{\iota}
    (c\rightarrow_{\iota}a\cdot_{\iota}c\cdot_{\iota}(b\cdot_{\iota}c)))$
\item The quadruple $(A_\iota,\leq,\to_\iota,\mathcal S_\iota)$ is an
  implicative algebra. This implicative algebra will be denoted as
  $\mathcal A_\iota$. \item In the commutative triangle below --that
  is just the decomposition of $\iota$ as an epi/mono--, all the
  morphisms are arrows in the category ${\bf IAc}$. \begin{equation*}
    \xymatrix{A\ar[rr]^{\iota}
      \ar[rd]_{\iota'}&&A\\ &A_\iota\ar[ru]_{\mathrm{inc}}&} \end{equation*} \end{enumerate} \end{theo} \begin{proof}
  For simplicity, we write $a \cdot b =ab$. It is clear that the
  elements $\ck_{\iota}$ and $\cs_{\iota}$ are in $\mathcal
  S_\iota$. \begin{enumerate} \item Given $a,b$ as above, we have that
    $\ck ab \leq a$ and $\ck a \leq b\to a$ and then $\ci (\ck a) \leq
    \iota(\ck a)\leq b \to_\iota a$, where we used the hypothesis that
    for all $d \in A$, $\ci d \leq \iota(d)$. Then by Lemma
    \ref{lema:asocas} $\nu \ci \ck a \leq b \to_\iota a$. Using
    adjunction and applying $\iota$ we deduce that $\iota(\nu \ci \ck)
    \leq a \to_{\iota} b \to_\iota a$. \item For all $a,b,c\in
    A_{\iota}$, $\cs a b c\leq a c (b c)$ and then $\cs a b c\leq a
    \cdot_{\iota}c \cdot_{\iota}(b \cdot_{\iota}c)$, where we used the
    fact that $ab \leq c_{\iota}(ab)=a \cdot_\iota b$. By the original
    adjunction in $A$, $\cs a b\leq c\rightarrow
    a\cdot_{\iota}c\cdot_{\iota}(b\cdot_{\iota}c)$, and then by
    applying $\iota$ and the definition of $\to_\iota$ we have that
    $\ci (\cs ab)\leq \iota(\cs a b)\leq c\rightarrow_{\iota}
    a\cdot_{\iota}c\cdot_{\iota}(b\cdot_{\iota}c)$, the first
    inequality comes from the hypothesis that for all $a \in A, \ci a
    \leq \iota(a)$. Hence: $\iota(\nu \ci (\cs a))\leq
    \iota(b\rightarrow (c\rightarrow_{\iota}
    a\cdot_{\iota}c\cdot_{\iota}(b\cdot_{\iota}c)))$ that can be
    deduced as before by using successively Lemma \ref{lema:asocas},
    the usual adjunction and the monotony of $\iota$. Similarly than
    before we deduce that $\ci (\nu \ci (\cs a))\leq
    b\rightarrow_{\iota} (c\rightarrow_{\iota}
    a\cdot_{\iota}c\cdot_{\iota}(b\cdot_{\iota}c))$ Now, if we apply
    twice Lemma \ref{lema:asocas} we have that: $\nu (\nu \ci (\nu
    \ci)) \cs a \leq \nu \ci (\nu \ci) (\cs a) \leq \ci (\nu \ci (\cs
    a))$ Then, the inequality $\iota(\nu (\nu \ci (\nu \ci)) \cs)\leq
    a\rightarrow_\iota (b\rightarrow_{\iota}
    (c\rightarrow_{\iota}a\cdot_{\iota}c\cdot_{\iota}(b\cdot_{\iota}c)))$
    is deduced by the adjunction property together with the monotony
    of $\iota$. \item To check that $\mathcal A_\iota$ is an
    implicative algebra we proceed as follows. The conditions
    concerning the properties of the operations in particular the
    commutation of the implication $\to_\iota$ with infinite meets are
    the content of Theorem \ref{lema:change}. The fact that the
    combinators $\ck, \cs$ of $\mathcal A_i$ are in the separator,
    follows from parts (1) and (2) above because they have to be
    larger than the corresponding combinators
    $\ck_{\iota}\,,\,\cs_{\iota}$. Notice that if for $a,b \in
    A_\iota$ the pair of elements $a,a \to_\iota b \in\mathcal S \cap
    A_\iota$, first we conclude that $a\to b$ is also in $\mathcal S$
    being larger than $a \to_{\iota} b$. As $\mathcal S$ is a
    separator, we deduce that $b \in \mathcal S \cap A_\iota$. Then,
    $\mathcal S_\iota$ is a separator. \item \begin{enumerate} \item
      We prove that the map $\iota':A \to A_\iota$ is computationally
      dense. We need to prove the following assertions: (i)
      $\iota'(\mathcal S) \subseteq \mathcal S \cap A_\iota$; (ii)
      $\exists r \in \mathcal S \cap A_\iota$ such that $\forall s \in
      \mathcal S\,,\,a \in A\,,\, r\cdot_\iota \iota'(s)\cdot_\iota
      \iota'(a) \leq \iota'(sa)$; (iii) For $X \subseteq A$,
      $\iota'(\bigcurlywedge X)=\bigcurlywedge_\iota(\iota'(X))$; (iv)
      There is a function $h : \mathcal S \cap A_\iota \to \mathcal S$
      and an element $t \in \mathcal S \cap A_\iota$ such that for all
      $s \in \mathcal S \cap A_\iota$, $t\cdot_\iota\iota'(h(s))\leq
      s$. Condition (i) is satisfied by definition and (iii) follows
      from the fact that $\iota$ is an Alexandroff operator. In order
      to prove (iv) take $h$ to be the inclusion map and
      $t=\ci_{\iota}= \cs_\iota \cdot_\iota\ck_\iota
      \cdot_\iota\ck_\iota$. For the proof of (ii) we proceed as
      follows: as $\iota(s) \leq s$ and $\iota(a) \leq a$ we deduce
      that $\nu \cdot \ci \cdot\iota(s) \cdot \iota(a) \leq \nu\cdot
      \ci \cdot s \cdot a \leq \ci\cdot (s \cdot a) \leq \iota(s\cdot
      a)$.  Using the adjunction we obtain that $\nu \cdot \ci \cdot
      \iota(s)\leq \iota(a)\rightarrow\iota(s \cdot a)$ and
      $\ci\cdot(\nu\cdot\ci\cdot\iota(s))\leq\iota(\nu\cdot\ci\cdot\iota(s))\leq
      \iota(a)\rightarrow_{\iota}\iota(s\cdot a)$. Using Lemma
      \ref{lema:asocas} and again the adjunction property we have that
      $\iota(\nu\cdot\ci\cdot(\nu\cdot\ci))\leq
      \iota(\iota(s)\rightarrow(\iota(a)\rightarrow_{\iota}\iota(s\cdot
      a)))=\iota(s)\rightarrow_\iota(\iota(a)\rightarrow_{\iota}\iota(s\cdot
      a))$. If we call $r=\iota(\nu\cdot\ci\cdot(\nu\cdot\ci)) \in
      \iota(\mathcal S)\subset \mathcal S\cap A_{\iota}$, then the
      inequality above becomes :
      $r\cdot_{\iota}\iota(s)\cdot_{\iota}\iota(a)\leq\iota(s \cdot
      a)$ by adjunction in $A_{\iota}$. \item To prove that the
      inclusion is a morphism of ${\bf IAc}$ (see Definition
      \ref{defi:morfimplalgsimpler} and\ref{defi:compdensemor}) we need
      to check that: (i) $\operatorname{inc}(\mathcal S \cap A_i)
      \subseteq \mathcal S$; (ii) $\exists r \in \mathcal S: \forall s
      \in \mathcal S \cap A_\iota\,,\, a \in A_\iota: r(sa) \leq sa$;
      (iii) The inclusion commutes with infinite meets, i.e. if $X
      \subseteq A_\iota$ then $\bigcurlywedge_{\iota}X=\bigcurlywedge
      X$. (iv) $\exists h: \mathcal S \to \mathcal S \cap A_\iota$ and
      $\exists t \in \mathcal S: \forall b \in \mathcal S\,,\, th(b)
      \leq b$. It is clear that the first condition is satisfied. For
      the second, just take $r=\ci \in \mathcal S$ --the combinator
      $\ci$-- and the third is obvious (see Theorem
      \ref{lema:change}). As we know that the interior operator is
      compatible with $\mathcal A$ if we take $h=\iota|_{\mathcal
        S}:\mathcal S \to \mathcal S \cap A_\iota$ ( i.e., $\iota$
      restricted to $\mathcal S$ with codomain $\mathcal S \cap
      A_\iota$) and $t=\ci$, we have that $\ci \iota(b) \leq \iota(b)
      \leq b$ for all $b \in \mathcal
      S$. \end{enumerate} \end{enumerate} \end{proof} 
      \item \label{item:structure}

The above constructions assemble in the following adjoint pair of
functors as illustrated in the diagram.
\[\xymatrix{{\bf IAc} \ar@/^2pc/[rr]^{T} &
  {\hspace*{-0.05cm}\mbox{\larger[1]$\,\,\,\,\,\,\,\perp$}}&
  \ar@/^2pc/[ll]^{V}{\bf Co_{\operatorname{imp}}(IAc)}} .\]
\begin{defi}\label{defi:basicfunctors}
  We define the following functors:\begin{enumerate} \item $T:{\bf
      IAc}\to {\bf Co_{imp}(IAc)} $ in
    objects $T(\mathcal A)=(\mathcal A, \operatorname{id})$; if $f:
    \mathcal A \to \mathcal B$ then $T(f)=f$; \item $U:{\bf Co_{imp}(IAc)} \to {\bf IAc}$
    in objects $U(\mathcal A,\iota)=\mathcal A$, for arrows
    $f:(\mathcal A,\iota) \to (\mathcal B,\kappa)$, $U(f)=f$; \item
    $V:{\bf Co_{imp}(IAc)} \to {\bf IAc}$ in objects $V(\mathcal
    A,\iota)=\mathcal A_\iota$ (see Theorem
    \ref{theo:implicativeiota}), for arrows $f:(\mathcal A,\iota)\to
    (\mathcal B,\kappa)$, $V(f)=f|_{A_\iota}: A_\iota \to B_\kappa$.
\end{enumerate} \end{defi} The fact that these constructions are functors is clear. For example, in the case of $V$, an arrow in the domain
category is an arrow $f:A \to B$ with the property that $\kappa f=f
\iota$. Hence, if $a \in A_\iota$ $\kappa(f(a))=f(\iota(a))=f(a)$ and
then $f(a) \in B_\kappa$. Once this restriction is established it is
clear that for such an $f$, $V(f)=\nu' f \operatorname{inc}$. As
$\nu'$ and $\operatorname{inc}$ are maps in ${\bf IAc}$ (see Theorem
\ref{theo:implicativeiota}) we deduce that $V(f)$ is a computationally
dense morphism. \begin{nota} The functor $T$ is what in
  \cite{kn:street} is called the identity functor or the inclusion
  functor. \end{nota}
\begin{theo} \label{theo:two adjunctions} In the notations of Definition \ref{defi:basicfunctors} we have the following adjunction $\, T \dashv
V$. \end{theo} \begin{proof} \begin{enumerate} 
    \item By a direct computation one can prove that
    the arrows defined for $(\mathcal A,\iota)\in {\bf
      Co_{\operatorname{imp}}(IAc)}$, $\mathcal B \in {\bf
      IAc}$: \[\varepsilon_{(\mathcal A,\iota)}: TV((\mathcal
    A,\iota))=(\mathcal A_\iota,\operatorname{id})\longrightarrow
    (\mathcal A,\iota)\quad;\quad \eta_{\mathcal B}: \mathcal B \to
    VT(\mathcal B)= \mathcal B,\] as
\[\varepsilon_{(\mathcal A,\iota)}= \operatorname{inc}: (\mathcal A_{\iota},\operatorname{id}) \to (\mathcal A,\iota)\quad;\quad
\eta_{\mathcal B}= \operatorname{id}_\mathcal B: \mathcal B \to
\mathcal B,\] satisfy the conditions that guarantee that the pair of
functors $(T,V)$ is an adjoint pair. For example, the fact that
$\varepsilon$ is computationally dense is the content of
Theorem\ref{theo:implicativeiota}. The fact that $\varepsilon$
commutes with the operators is simply the assertion $\iota
\operatorname{inc}=\operatorname{inc} \operatorname{id}$, which is
valid in $\mathcal A_{\iota}$. It is clear that the identity map is a
computationally dense morphism. \end{enumerate} \end{proof}

\section{From implicative comonads to triposes, revisiting the equivalence results}\label{section:IAtoTripos}
\item We summarize the equivalence
results of \cite{kn:ocar,kn:ocar2} as: \begin{enumerate} \item[a. ]
  Given the AKS named $\mathcal K$ we consider the IAs: $A(\mathcal
  K),\,A_\bullet(\mathcal K)$ and the IOCA: $A_\perp(\mathcal K)$.
  Then, the triposes ${\bf H}_{A(\mathcal K)}$, ${\bf H}_{A_\bullet(\mathcal K)}$ and ${\bf H}_{A_\perp(\mathcal K)}$ and ,
  are equivalent. \item[b. ] Given an implicative algebra $\mathcal B$
  we consider the associated IAs: $A(K(\mathcal B))\,, A_\bullet(K(\mathcal B))$, and the IOCA: $A_\perp(K(\mathcal B))$. The
  corresponding triposes ${\bf H}_{\mathcal B}\,,{\bf H}_{A(K(\mathcal
    B))}\,, {\bf H}_{A_\bullet(K(\mathcal B))}\,,{\bf
    H}_{A_\perp(K(\mathcal B))}$ are equivalent. \end{enumerate} For
the definitions of the structures mentioned in the above equivalence
assertions, the reader can consult Paragraph \ref{item:implvsfoca} and
also \cite[Section 4]{kn:ocar}. Moreover, some recollections will
appear in the treatment that follows. The definition of the map
$A_\perp$ from AKSes to IOCAs that is used above, is briefly recalled
in Observations \ref{obse:perpioca} and \ref{obse:mainexam}. For the
definition of $A_\bullet$ as a map from AKSes to IAes, see the
considerations of the next paragraph (for more details consult
\cite{kn:ocar2} where this construction was introduced).

\item\label{item:adaptation} Our
primary purpose, in this section and the next, is to reinterpret the
 equivalence results mentioned above, as special cases of the
categorical considerations that we develop in the present paper.

\smallskip \noindent The reader probably noticed that not all the
structures considered above are in the category of implicative
algebras. Indeed, the ones that involve Streicher's
construction (that correspond to the closure operator given by double
perpendicularity, i.e. the ones centered around the construction named
as $A_\perp$), belong to the realm of IOCAs and will be treated in
Section \ref{section:implocas}.  In this section,  we concentrate in the
study of the equivalences in the family of IAs.

\smallskip \noindent The triposes appearing in (a.) are the
\emph{Krivine's, Streicher's and bullet} triposes respectively. The
equivalence of Krivine's and bullet triposes fit squarely in the
theory we develop being particular cases of the general implicative
algebras $U(\mathcal A,\iota)=\mathcal A$ and $V(\mathcal
A,\iota)=\mathcal A_\iota$ that are proved to have equivalent
associated triposes (when the interior operator satisfies a certain
additional compatibility condition) in Theorem
\ref{theo:equivalenciaenalexandroff}. Indeed, in the case that
$\mathcal A=A(\mathcal K)$, for $\mathcal K$ an arbitrary AKS, and
$\iota$ is the Alexandroff approximation of Streicher's double
perpendicular operator (see Theorem \ref{theo:existencealaprox})
corresponding to the pole of the AKS, we have that $\mathcal
A_\iota=A_\bullet(\mathcal K)$ (compare with the original version
\cite[Paragraph XXVIII, Theorem 6.6]{kn:ocar2}).

\smallskip \noindent Concerning the equivalences of (b.), it is clear
that once (a.) is established, all we need is to prove the assertion
in the case (for example) of $A(K(\mathcal B))$ and $\mathcal B$ for
an arbitrary implicative algebra $\mathcal B$. This is the content of
Theorem \ref{theo:equivalenciaenadjuntos} that corresponds
--indirectly-- to \cite[Theorem 6.13]{kn:ocar2}. We refer the reader
to the considerations at the beginning of Paragraph \ref{item:ingrain}
for the reasons why we pay special attention to results like the
theorem mentioned above.

\item\label{item:theresults} We start by proving a general result
  that establishes the connection between the categorical viewpoint
  adopted here and the proofs of the equivalence results presented in
  our previous work.
  \begin{theo}\label{theo:general} Let $\mathcal A$
    and $\mathcal B$ be IAs and assume that we have two maps
    $f:A \to B$ and $g:B\to A$ with the following properties.
    \begin{enumerate} \item \label{item:thereare1}
      {\tt Properties of $f: A \to B$.}
      \begin{enumerate}
          \item \label{item:compseparator1}$f(\mathcal S_\mathcal A)
            \subseteq \mathcal S_\mathcal B$;
          \item $\exists s\in \mathcal S_\mathcal B: \forall (a \vdash
            a'),\, \,\, s \leq f(a\to_A a')\to_{B} f(a)\to_B f(a')$;
\item The map $f:A \to B$ is monotonic.
      \end{enumerate}
\medskip
    \item {\tt Properties of $g:B \to A$ -and $f$.}
      \begin{enumerate}
        \item $g(\mathcal
      S_\mathcal B) \subseteq \mathcal S_\mathcal A$.
    \item
      \begin{enumerate}\label{item:existste1}\item $\exists t\in
          \mathcal S_\mathcal A:\forall (f(a) \vdash f(a')),
          \,\,\,t\leq g(f(a) \to_B f(a')) \to_A a \to_A
          a'$; \item \label{item:comdense1}$fg(b)\to_B b' = b \to_B
          b'$ for every $b,b'\in
          B$.\end{enumerate} \item The map $g: \mathcal B \to \mathcal A$ is monotonic. \end{enumerate} \end{enumerate}
In that situation for all sets $I$ the maps given by post
    composition with $f:A \to B$ $f_{*I}:=f-:{\bf H}_{\mathcal A}(I)
    \to {\bf H}_{\mathcal B}(I)$ are monotonic, reflect the order and
    are essentially surjective. Then $f_{*I}$ is an equivalence of
    preorders.
  \end{theo}
  \begin{proof}Crearly, $f_{*I}$ is a map
    between indexed preorders. \textit{Monotonic:} Assume that we have
    $\varphi,\psi \in A^I$ with $\varphi\vdash_I\psi$ then $\exists
    r\in \mathcal S_\mathcal A: \forall i\in I\,\,\, r \cdot
    \varphi(i)\leq \psi(i)$ i.e., $\forall i\in I\,,\,
    r\leq\varphi(i)\to \psi(i)$. Since $f$ is monotonic we have that
    $\forall i\in I\,,\,f(r)\leq f(\varphi(i)\to \psi(i))$. From that
    and using (1) condition (b), we deduce that for some $r \in
    \mathcal S_{\mathcal A}\,,\,s \in \mathcal S_{\mathcal B}$ :\,
    $s\cdot f(r)\leq s\cdot f(\varphi(i)\to \psi(i))\leq
    f(\varphi(i))\to f(\psi(i))$ which implies that $\exists \ell \in
    \mathcal{S}_\mathcal B: \forall i\in I\,\,\, \ell\cdot
    f(\varphi(i))\leq f(\psi(i))$, which means that $f\varphi \vdash_I
    f\psi$ i.e., that $f_{*I}$ is monotonic. \textit{Reflective:} Now
    we want to prove that it reflects the order. Suppose that
    $f\varphi\vdash_I f\psi$, then $\exists \ell\in
    \mathcal{S}_\mathcal B$ such that $ \forall i\in I\,\,\, \ell\cdot
    f(\varphi(i))\leq f(\psi(i))$, hence $\forall i\in I\,,\, \ell\leq
    f(\varphi(i))\to f(\psi(i))$ and using that $g$ is monotonic we
    obtain that $\forall i\in I\,\,\,g(\ell)\leq g(f(\varphi(i))\to
    f(\psi(i)))$. Now using (2) condition (b,i), we can guarantee that
    $\exists t\in \mathcal S_\mathcal A: \forall i\in I\,\,\, t\cdot
    g(\ell)\leq t\cdot g(f(\varphi(i))\to f(\psi(i)))\leq
    \varphi(i)\to \psi(i)$ which means that $\exists t'\in
    \mathcal{S}_\mathcal A: \forall i\in I\,\,\, t'\leq \varphi(i)\to
    \psi(i)$ i.e., $\varphi\vdash_I\psi$. Concerning the essential
    surjectivity we proceed as follows. Take $\gamma \in B^I$ and
    consider $g\gamma \in A^I$. We want to show that $fg\gamma \cong
    \gamma$. For any $b \in B$ we have that $\ci \leq fg(b) \to_B
    fg(b)$ and applying (2) condition (b,ii), we obtain that $\ci \leq
    b \to_B fg(b)$ for every $b\in B$. This means that
    $\operatorname{id} \vdash_B fg$ and then that $\gamma
    \vdash_Ifg\gamma$. Conversely, using the fact that $\ci \leq b
    \to_B b =f(g(b))\to_B b$ (c.f. hypothesis (2) condition (b,ii)),
    we deduce that $fg \vdash_B \operatorname{id}$ that implies
    $fg\gamma \vdash_I \gamma$. Putting together both results we
    obtain that $f_{*I}(g\gamma)\cong \gamma$. \end{proof}

  \begin{obse} Notice that of the hypothesis of the above Theorem
    \ref{theo:general} for the map $f: \mathcal A \to \mathcal B$, the
    first two conditions (1,a) and (1,b) are the same than the first
    two conditions of the definition of applicative morphism of
    implicative algebras (Definition
    \ref{defi:morfimplalgsimpler}). The third condition (1,c) is a
    weakening as we are changing the condition of compatibility with
    infinite meets by its consequence the condition of the monotony
    for the function $f$.  Comparing the hypothesis (2) for the map
    $g$ with the condition for the map $h$ in Definition
    \ref{defi:compdensemor} of computationally dense morphism, we
    observe that conditions (2,a) and (2,c) appear also as
    conditions  for $h$ and if the hypothesis (2) part (b,ii)
    holds, reasoning as in the proof of the theorem, we obtain that
    $\ci \leq b \to b = fg(b) \to b$ for all $b \in \mathcal
    S_B$. This is exactly the condition that the map $h$ in the
    definition of computationally dense morphism has to satisfy.  It
    is then clear, that except for the monotony of $f$ (that is weaker
    than the preservation of infinite meets) any morphism that is in
    the hypothesis of the above theorem is a computationally dense
    morphism of IAs.
  \end{obse}
  In the next Corollary, we show that two IAs related by a morphism of
  the type considered above, induce equivalent triposes. This will
  imply the equivalence of Krivine's and the bullet triposes.

  \begin{coro}\label{coro:equivalenciaconlaiota} Let $\mathcal A$
    and $\mathcal B \in {\bf IA}$ and assume that the following
    conditions for the implication and the separators are
    satisfied:
    \begin{enumerate}
    \item There is an interior operator
      --a comonad-- $\iota:B \to B$ satisfying the following
      properties: $A=\{b \in B: \iota(b)=b\}$, $a \to_A a'=\iota(a
      \to_B a')$ and $b \to_B b'=\iota(b)\to_B b'\,,\, a,a' \in A\,,\, b,b' \in B$.
    \item
      $\iota(\mathcal S_\mathcal B) = \mathcal S_\mathcal B\cap A =
      \mathcal S_{\mathcal A}$. \end{enumerate} In this situation for
    all sets $I$ the injections ${\bf H}_{\mathcal A}(I) \subseteq {\bf
      H}_{\mathcal B}(I)$, preserve and reflect the order and are essentially
    surjective. Hence, the injection is an equivalence of
    preorders. \end{coro} \begin{proof} We prove that the set
    theoretical maps $f:A\to B\,,\,g:B \to A$ defined as
    $f:=\operatorname{inc}$ and $g:=\iota$ satisfy the conditions of
    Theorem \ref{theo:general}.

  \begin{itemize}
\item[(1.a,2.a)] The conditions $f(\mathcal S_\mathcal A) \subseteq
  \mathcal S_\mathcal B$ and $g(\mathcal S_\mathcal B)
  \subseteq\mathcal S_\mathcal A$ are clear. The first is simply
  $\mathcal S_{\mathcal A} \subseteq \mathcal S_{\mathcal B}$ and the
  second is $\iota(\mathcal S_\mathcal B) \subseteq \mathcal
  S_\mathcal A$ that follows directly from (in fact it is equivalent
  with) the hypothesis (2).
  \item[(1.b)] $\exists s\in \mathcal S_\mathcal B: \forall (b \vdash
    b') \,\,\,\,s\cdot f(b\to_A b')\leq f(b)\to_B f(b')$. To prove
    this assertion, consider $s=\ci$ and we have $\ci (b \to_A b')
    \leq b \to_A b' =\iota(b \to_B b')\leq b \to_Bb'$;
\item[(2.b.i)] $\exists t\in \mathcal S_\mathcal A: \forall (f(a)
  \vdash f(a')) \,\,\,\,t\cdot g(f(a) \to_B f(a'))\leq a \to_A
  a'$. Consider $t=\ci$ and we have $\ci \cdot
  \iota(\operatorname{inc}(a) \to_B \operatorname{inc}(a'))
  \leq\iota(\operatorname{inc}(a) \to_B \operatorname{inc}(a')) = a
  \to_A a'$;
\item[(2.b.ii)] $fg(b)\to_B b' = b \to_B b'$ since we have
  $\iota(b)\to_B b' = b \to_B b'$ for every $b,b'\in B$;
\item[(1.c,2.c)] The conditions of monotony are clearly satisfied. 
  \end{itemize}
\end{proof}
  \begin{coro}\label{coro:fyginversas} Let
    $f:{\mathcal A}\to{\mathcal B}$ be an invertible arrow in ${\bf
      IA}$ with inverse $g$, then it induces an equivalence of ${\bf
      H}_{\mathcal A}$ and ${\bf H}_{\mathcal B}$ as indexed HPO.
\end{coro}
\begin{proof} We prove
  that the maps $f$ and $g$ satisfy the conditions necessary to
  guarantee the equivalence of indexed preorders appearing in Theorem
  \ref{theo:general}. Regarding the map $f$ it is clear that being a
  morphism of implicative algebras it satisfies the first three
  conditions of the theorem. For $g$ the condition (2,b,i) is deduced
  as follows: from the fact that $g$ is applicative we deduce that
  there is a $t \in \mathcal S_{\mathcal A}$ such that, for all $f(a)
  \vdash f(a')$, $t \cdot g(f(a) \to_B f(a')) \leq g(f(a)=a \to_A
  g(f(a'))=a'$. Condition (2,b,ii) follows from the fact that $f$ and
  $g$ are set--theoretical inverse. Conditions (2.a) and (2.c) are
  direct consequences of the properties of applicative morphisms.
\end{proof}

\begin{rema}\label{rema:relevant} Regarding the proof
 of the preceding corollary and for future reference, it is worth
 mentioning that we did not use fully the hypothesis that the maps $f$
 and $g$ are morphisms in the category ${\bf IA}$ as we could do with
 the substitution of the condition that they preserve infinite meets, by the condition that
 the two maps are monotonic --that is weaker.
\end{rema}

As we mentioned before, the result that follows is a generalization of
the equivalence between Krivine's tripos and ``bullet''
tripos.

\begin{theo}\label{theo:equivalenciaenalexandroff} Consider
  the functors $V,U:{\bf Co_{\operatorname{imp}}(IA)} \to{\bf IA}$ and
  let $(\mathcal A,\iota) \in {\bf Co_{\operatorname{imp}}(IA)}$ be an
  object with $\iota$ a strongly compatible comonad. Then, the triposes
  associated with $V(\mathcal A,\iota)$ and $U(\mathcal A,\iota)$, are
  equivalent.
\end{theo}
\begin{proof} The result follows immediately
  from Corollary
  \ref{coro:equivalenciaconlaiota}.
\end{proof} \item\label{item:ingrain}

Next, we address the categorical treatment of the assertions about the
equivalence of triposes that we recalled in (b.) at the beginning of
this section. As we mentioned before, it is enough to deal with the
situation of the equivalence between the triposes associated with the
original IA named $\mathcal B$ and $A(K(\mathcal B))$. One of the main
goals in the series of papers: \textit{Ordered combinatory algebras
  and realizability} (c.f. \cite{kn:ocar}), \textit{Realizability in
  ordered combinatory algebras with adjunction} (c.f. \cite{kn:ocar2})
and the current one, is to ingrain the theory of classical
realizability in the field of ordered combinatory algebras (in its
many guises ranging from IOCAs to IAs). With that purpose in mind
--amounting to the ``algebraization of realizability''-- , we proved
that the models (triposes) that we called Krivine's and Streicher's
models, could be produced --up to equivalence-- by conveniently chosen
IOCAs or IAs. This was settled in the mentioned papers, in particular
in \cite[Theorem 5.16]{kn:ocar} and \cite[Theorem 6.13]{kn:ocar2} and
the commentaries that follow therein. In the next theorem, we show that the map that induces the required
equivalence is the counit of the
adjunction $A \dashv K$ (Theorem \ref{theo:mainadj}) .
  
First, we collect some notations and results. The functors $A:{\bf
  AKS}\to {\bf IA}\,,\,K:{\bf IA}\to {\bf AKS}$ are an adjoint pair
with counit given for $\mathcal B \in {\bf IAc}$ as the map
$\varepsilon_{\mathcal B}:\mathcal P(B) \to B$ defined for every $P
\subseteq B$ as $\varepsilon_{\mathcal B}(P)=\bigcurlywedge P$. Recall
(see Paragraph \ref{basicdiagram} item \eqref{item:AK}) that the basic
set supporting $A(K(\mathcal B)) \in {\bf IA}$ is $\mathcal P(B)$ with
the following structure of IA. If $\mathcal B=(B,\leq,\to,\mathcal
S_{\mathcal B})$, then $A(K(\mathcal B))=(\mathcal P(B), \leq,
\leadsto,\mathfrak S)$
where:

\begin{enumerate}
\item For $C,D \subseteq B\,,\,C \leq D:=C
  \supseteq D$;
\item For $C,D \subseteq B\,,\,C \leadsto D:=\{c \to
  d:c \leq \bigcurlywedge C, d \in D\}$;
\item If we call $\mathfrak k ,\,\mathfrak s$ and
  $\ck,\,\cs$, the basic combinators of $AK(\mathcal B)$ and $\mathcal
  B$ respectively, we have that $\mathfrak k \subseteq
  {\uparrow}\,\ck\,$ and $\mathfrak s \subseteq {\uparrow}\,\cs\,$;
\item $\mathfrak S=\{C \subseteq B: \bigcurlywedge C \in \mathcal
  S_B\}$.
\end{enumerate}

\medskip

The maps $f$ and $g$ of the theorem that follows have already been
considered in the paper (for example in Theorem \ref{theo:mainadj})
and below we put their properties are together to produce the map that
realizes the procured equivalence.  Notice that the map called $g$
 just above, is the map that was introduced to guarantee the
computational density of $\varepsilon_{\mathcal B}$ in the above
mentioned theorem (see part (1b) of its proof).
\begin{theo}\label{theo:equivalenciaenadjuntos}
  For each $\mathcal B \in {\bf IAc}$ the counit of the adjunction $A
  \dashv K\,,\,\varepsilon_\mathcal B: \mathcal P(B) \to B$, induces
  an equivalence of the triposes associated to $A(K(\mathcal B))$ and
  $\mathcal B$.
\end{theo}

\begin{proof} We deduce the result of the
  equivalence of the tripos associated to $A(K(\mathcal B))$ and
  $\mathcal B$ by a direct application of Theorem \ref{theo:general}
  using the maps $f:=\varepsilon_\mathcal B$ and $g:B\to \mathcal
  P(B)$ defined as $g(b)={\uparrow}b$. We check the hypothesis of
  Theorem \ref{theo:general} for the maps $f$ and $g$. The hypothesis
  (1) that involves only the map $f$, are satisfied as a consequence
  of the fact that $f=\varepsilon_{\mathcal B}$ is a morphism in ${\bf
    IAc}$ (see Theorem \ref{theo:mainadj}). It is clear that
  $\varepsilon_{\mathcal B}({\uparrow}a)=a$, hence
  $fg=\operatorname{id}$ and the condition (2.b.ii) follows
  trivially. The fact that $g$ sends the separator into the separator,
  as well as its monotony (conditions (1) and (3)), are proved in
  Theorem \ref{theo:mainadj}. In our case the condition (2.b.ii) of
  Theorem \ref{theo:general} means that there is an element $t \in
  \mathcal S_{\mathcal P(B)}$ with the property that $t \cdot
  ({\uparrow}(\bigcurlywedge P \to \bigcurlywedge Q)) \supseteq P
  \leadsto Q$, for all $P,Q$ such that $\varepsilon_{\mathcal B}(P)
  \vdash \varepsilon_{\mathcal B}(Q)$. As $\bigcurlywedge P \to \inf
  Q=\inf(P \leadsto Q)$\,(Lemma \ref{lema:equalinf})\, the requirement
  for $t$ becomes $t \cdot {\uparrow}\inf(P \leadsto Q) = t \cdot
  \overline{(P \leadsto Q)} \supseteq P \leadsto Q$. The equality
  ${\uparrow}\inf(P \leadsto Q) = \overline{(P \leadsto Q)}$\, is
  obtained by a direct computation that is performed for example in
  \cite[Lemma 5.15]{kn:ocar}).  Hence, if we take $t=\cI$ we have that
  $\cI \cdot \overline{(P \leadsto Q)} \supseteq \cI \cdot (P \leadsto
  Q) \supseteq P \leadsto Q$.
  
\end{proof}

\section{The category of implicative ordered combinatory algebras}
\label{section:implocas}
\item In this section, and following
the path already established for IAs, we start by defining the
category of implicative ordered combinatory algebras, with objects the
class of IOCAs and the arrows defined below. As we mentioned before,
this becomes necessary in order to put Streicher's constructions (see
\cite{kn:streicher}) in a more comprehensive categorical
perspective. As Streicher's proposals do not involve the consideration
of implicative algebras, we have to step aside from the categories
introduced heretofore and deal with contexts where application and
implication do not satisfy the \emph{full adjunction property}. To be
able to work in that situation it is necessary to introduce the
so-called \emph{adjunctor combinator} (see \cite{kn:ocar,kn:ocar2}
for the scrutiny of this situation). In order to avoid repetitions, in
the definition of the categories of IOCAs and in the proofs of the
main results, we adopt a rather expeditious style omitting
 many details as the proofs and
constructions are similar to the ones performed for IAs.

\begin{obse}\label{obse:perpioca}
  The definition of
  implicative ordered combinatory algebra (a.k.a. IOCA) can be
  obtained by slightly changing the conditions for a FOCA as appears
  in Definition \ref{example:ocas}. Instead of the full adjunction
  condition, we require the existence of an element
  $\ce \in\Phi$\, such that for all $a,b,c \in A$ one of the following
  pair of equivalent conditions are
  satisfied:
  \begin{equation}\label{eqn:iocaunitcounit}
\begin{aligned} a \leq b\to c \Rightarrow ab \leq c \quad&,\quad ab \leq c\Rightarrow \ce a \leq b \to c \\ (b \to c)b \leq c \quad&,\quad
  \ce a \leq b \to ab. \end{aligned}
  \end{equation}
  This subject is treated in detail in \cite{kn:ocar2} where the concept is defined.

  The main example of IOCA is based in Streicher's construction in \cite{kn:streicher} and is briefly recalled in \ref{obse:mainexam}. For more information about this construction consult also \cite{kn:ocar}.
\end{obse}

\medskip \item\label{item:introd} Adapting Definition
\ref{defi:morfimplalgsimpler}, we define the notion of
\emph{computationally dense} (\emph{applicative}) morphism of IOCAs
(compare also with the results of Lemma
\ref{lema:morfimplalgsimpler}).
\begin{defi}\label{defi:morfioca}

Let $\mathcal A=(A,\leq, \cdot,\to,\Phi_\mathcal A)$ and $\mathcal
B=(B,\leq, \cdot,\to,\Phi_\mathcal B)$ be two implicative ordered
combinatory algebras. A \emph{computationally dense morphism} $f:
\mathcal A \to \mathcal B$ is a set-theoretical function
$f:A\rightarrow B$\, such that: \begin{enumerate} \item
  $f(\Phi_{\mathcal A})\subseteq \Phi_{\mathcal B}$; \item $\exists t
  \in \Phi_{\mathcal B}: \forall s\in \Phi_{\mathcal A}, a\in A,
  tf(s)f(a)\leq f(sa)$; \item For all $P \subseteq A$ we have that
  $f(\bigcurlywedge_A P)=\bigcurlywedge_B f(P)$,
  i.e. $f\left(\bigcurlywedge_A\{x: x \in P\}\right)=\bigcurlywedge_B
  \{f(x): x \in P\}$; \item There is a monotonic function
  $h:\Phi_{\mathcal B}\rightarrow \Phi_{\mathcal A}$ such that, $f
  h\vdash_{\Phi_{\mathcal B}}\operatorname{id}_{\Phi_{\mathcal
      B}}$. This last condition is equivalent to the following:
  $\exists t \in \Phi_{\mathcal B}:\forall b \in
  \operatorname{\Phi_{\mathcal B}}, tf(h(b)) \leq b$. \end{enumerate}
In case that $f$ satisfies only the first three conditions we say that
it is an applicative morphism of IOCAs.
\end{defi} \begin{obse}\label{iocaimplia} \begin{enumerate} \item Recall that in the same manner than for IAs, the relation $a \vdash b$ for $a,b \in A$ ($A$ a IOCA) can be characterized by any of the following two conditions:
  (i) $\exists s \in \Phi$ such that $sa \leq b$; (ii) $a \to b \in
    \Phi$.  \item In Lemma \ref{lema:morfimplalgsimpler} that is valid
    for IAs, we characterized morphisms that satisfy a condition like
    (2) above, in terms of the implication rather than the
    application. Below we discuss similar properties for the case of
    IOCAs. Let $\mathcal A$ and $\mathcal B$ be IOCAs and consider the
    following condition for a function $f:A \to B$: \[\exists t \in
    \Phi_{\mathcal B}, \forall a,a', \textrm{if}\,\, a \vdash a', \textrm{then } t f(a
    \to a')\leq f(a)\to f(a')\quad\quad (2'),\] and notice that in the
    context of Definition \ref{defi:morfioca}, condition (2) implies
    condition $(2')$ but in general the second is
    weaker. \begin{enumerate} \item Assuming $(2)$, if we have that $a
      \vdash a'$, $a \to a' \in \mathcal S_\mathcal A$ and then there
      is a $\ell$ such that $\ell f(a \to a')f(a) \leq f((a \to a')a)
      \leq f(a')$ by \eqref{eqn:iocaunitcounit}. Then, $\ce (\ell f(a
      \to a')) \leq f(a) \to f(a')$. Hence we deduce (using Lemma
      \ref{lema:asocas} as before), that there is an element $t \in
      \Phi_\mathcal B$ such that for all $a \vdash a'$, $tf(a \to a')
      \leq f(a) \to f(a')$. \item Assuming $(2')$, if we take $s \in
      \Phi_{\mathcal A}$ and $a \in A$, it follows by the
      half--adjunction condition \eqref{eqn:iocaunitcounit} that $\ce
      \!s \leq a \to sa$ and then that $a \to sa \in \Phi_\mathcal A$.
      Hence, in accordance with condition $(2')$, there is an element
      $t \in \Phi_\mathcal B$ such that $tf(\ce \!s) \leq tf(a \to
      sa)\leq f(a) \to f(sa)$ for all $a,s$ as above. Using again
      condition \eqref{eqn:iocaunitcounit}, we deduce that for the
      element $t \in \Phi_\mathcal B$ we have that $tf(\ce \!s) f(a)
      \leq f(sa)$ for all $s \in \Phi_\mathcal A$ and $a \in
      A$. \end{enumerate} Thus we obtain that $(2')$ implies a slight
    weakening of condition (2). \item In the case of an IA, that can
    be viewed as a IOCA (with $\ce=\ci$), it is
    clear that a computationally dense morphism of IAs is also a
    computationally dense morphism of IOCAs and vice versa. The same
    holds for applicative morphisms. \item The composition of two
    computationally dense morphisms of IOCAs is computationally dense,
    and the same holds for applicative morphisms. This assertion can
    be proved in a similar manner than for implicative
    algebras. \end{enumerate} \end{obse}
\begin{defi} We define as ${\bf IOCA}$ and ${\bf IOCAc}$ the categories that have objects the class of IOCAs and as arrows, the first one
the applicative morphisms of IOCAs and the second one the
computationally dense morphisms. \end{defi} \begin{obse} Given
  $\mathcal A \in {\bf IOCA}$ we can define the indexed Heyting
  preorder ${\bf H}_\mathcal A$ in the same manner than for
  IAs--eventually taking care of the combinators that in this
  case should be modified by applications of the adjunctor $\ce$
  whenever it is necessary in the computations and proofs. Moreover,
  this indexed Heyting preorder is a tripos (see \cite[Definition 3.16,
    Theorem 3.17]{kn:ocar2} for these assertions).
\end{obse} \item\label{item:nextweshow} Regarding the missing points of
the assertion of item (a.) at the beginning of Section
\ref{section:IAtoTripos}, we need to prove that the triposes
associated to the IOCAs $A_\perp(\mathcal K)$ and $A(\mathcal K)$, are
equivalent \,(recall \cite[Definition 5.10]{kn:ocar} or
\cite[Definition 6.1]{kn:ocar2}). This is the content of \cite[Lemma
  5.5]{kn:ocar} and its proof in the current set up follows after a
few preparations, that begin with the reformulation --without proofs--
of Theorem \ref{theo:general} and Corollary
\ref{coro:equivalenciaconlaiota}. \begin{theo}\label{theo:generaliocas}
  Let $\mathcal A$ and $\mathcal B$ be IOCAs and assume that we have two maps $f:A \to B$ and $g:B \to A$ with the following properties.
  \begin{enumerate} \item \label{item:thereare2}{\tt Properties of
      $f:A \to B$}.
    \begin{enumerate} \item $f(\Phi_{\mathcal A}) \subseteq
      \Phi_{\mathcal B}$;
    \item \label{item:existsese}$\exists s\in
      \Phi_\mathcal B: \forall (a \vdash a')\,\, \,\, s\cdot f(a\to_A
      a')\leq f(a)\to_B f(a')$;
    \item The map $f:A \to B$ is monotonic.
    \end{enumerate}
  \item {\tt Properties of $g:B \to A$ -and $f$}.
    \begin{enumerate}
\item $g(\Phi_{\mathcal B}) \subseteq \Phi_{\mathcal A}$;
\item \begin{enumerate}
\item \label{item:existste}$\exists t\in \Phi_\mathcal
      A: \forall (f(a) \vdash f(a'))\,\, \,\,t\cdot g(f(a) \to_B
      f(a'))\leq a \to_A a'$;
    \item \label{item:comdense}$fg(b)\to_B b' = b \to_B b'$ for every
      $b,b'\in B$.
\end{enumerate}
  \item \label{item:compseparator} The map $g:B \to A$ is monotonic.
    \end{enumerate}
\end{enumerate}
    In that situation
  for all sets $I$ the maps given by post composition with $f:A \to B$
  $f_{*I}:=f-:{\bf H}_{\mathcal A}(I) \to {\bf H}_{\mathcal B}(I)$ are
  monotonic, reflect the order and are essentially surjective. Then
  $f_{*I}$ is an equivalence of preorders.
  \end{theo}

  The following considerations are related to the Definition
  \ref{defi:basicfunctors} and also to Corollary
  \ref{coro:equivalenciaconlaiota}. In a different manner that
  therein, here we have to take into account the fact that the
  application is not determined by the
  implication. \begin{defi}\label{defi:almostbasicfunctors} Let
    $\mathcal B$ be a IOCA with base space $B$ and assume that there
    is an interior operator $\iota: B \to B$ --a comonad-- with the
    properties that $\iota(b)\to_\mathcal B b' = b \to_\mathcal B b'$
    and $\iota(\Phi_\mathcal B) \subseteq \Phi_\mathcal B$. Call
    $B_{\iota}$ the set of $\iota$--open subsets of $B$ and define
    $\to_\iota:B_\iota \times B_\iota \to B_\iota$ as $b \to_\iota b'
    := \iota (b \to_\mathcal B b')$ and $b \cdot_{\iota}b':= \iota(b
    \cdot_{\mathcal B}b') $.
\end{defi}
\begin{obse} \label{obse:mainexam} The main example of the above situation is the case
  of the IOCAs, defined originally by Streicher in \cite{kn:streicher}
  and that were one of the main topics of consideration in
  \cite{kn:ocar} and \cite{kn:ocar2}. They are associated to an AKS
  $\mathcal K$, and denoted as $A_\perp(\mathcal K)$: the basic
  ordered set of this IOCA is $\mathcal P_\perp(\Pi)=\{P \subseteq
  \Pi: ({}^\perp P)^\perp =P\}$ --i.e. the subsets of $\Pi$ fixed by
  the comonad $P \mapsto ({}^\perp P)^\perp :\mathcal P(\Pi)
  \stackrel{\iota_0}{\longrightarrow} \mathcal P(\Pi)$ associated to
  the polarity of $\mathcal K$; the order is the opposite inclusion
  order, and the filter and basic combinators are defined in terms of
  the set $\operatorname{QP}$ and the elements $\cK$ and $\cS$ of
  $\Pi$ as shown in Paragraph \ref{item:AKconstruction} of the current
  paper (see also Definition 4.3 and 4.15 in \cite{kn:ocar2}). If we
  name as $\fpush$ the push operation in the original AKS, the
  operations in the IOCA are given as:
  \begin{enumerate}
    \item Application map $\cdot_{\iota_0}:\mathcal P_{\perp}(\Pi)
      \times \mathcal P_{\perp}(\Pi) \to \mathcal P_{\perp}(\Pi)$: \[P
      \cdot_{\iota_0} Q=\big({}^\perp\big\{\pi \in \Pi: \forall t
      \perp Q\,,\, \fpush(t,\pi) \in P\}\big)^\perp=\iota_0(\{\pi \in
      \Pi: \fpush({}^\perp Q,\pi) \in P\});\]
\item Implication map $\to_{{\iota_0}}:\mathcal P_{\perp}(\Pi)
  \times \mathcal P_{\perp}(\Pi) \to \mathcal P_{\perp}(\Pi)$:
  \[P \to_{\iota_0} Q=\big({}^\perp\big\{\fpush(t,\pi):
      \forall t \perp P\,,\,\pi \in
      Q\}\big)^\perp=\iota_0(\fpush({}^\perp P,Q)).\]
  \end{enumerate}
  In \cite[Paragraph \large{\sf{XIX}}]{kn:ocar2} the reader can find
  more details about this construction. The proof that in this manner
  we obtain a IOCA appears in \cite[Paragraph
    \large{\sf{XXV}}]{kn:ocar2}, as well as in \cite[Theorem
    2.13]{kn:ocar}, and are inspired in the original Streicher's
  proof: \cite[Lemma 5.4]{kn:streicher}.

  All the considerations of this section as well as part of the
  motivation for the main hypothesis of Theorem 8.1 and the rest of the considerations in Paragraph
  \ref{item:theresults} are related to the fact that for the case of
  $\mathcal P_\perp(\Pi)$ the operation of application and implication
  are no longer adjoint.
\end{obse} 

 \begin{coro}\label{coro:equivalenciaconlaiota2} Let $\mathcal B$ be a
  IOCA equipped with an interior operator as in Definition
  \ref{defi:almostbasicfunctors}. Assume moreover that there is an
  application map $\cdot_\iota: B_\iota \times B_\iota \to B_\iota$
  with the property that the quadruple $\mathcal B_\iota:=(B_\iota,
  \leq|_{B_\iota}, \cdot_\iota, \to_\iota, \iota(\Phi_\mathcal
  B)=\Phi_\mathcal B \cap B_\iota:=\Phi_{\mathcal B_\iota})$ is an
  $IOCA$. In this situation for all sets $I$ the injections ${\bf
    H}_{\mathcal B_\iota}(I) \subseteq {\bf H}_{\mathcal B}(I)$,
  preserve and reflect the order and are essentially
  surjective. Hence, the injection is an equivalence of
  preorders. \end{coro} \begin{theo}\label{theo:triposstrequivkriv}
  Let $\mathcal K \in {\bf AKS}$ and consider the associated IOCAs,
  $A_\perp(\mathcal K)$ and $A(\mathcal K)$. Then, the inclusion of
  the first into the second, induces an equivalence between the
  associated triposes ${\bf H}_{A_\perp(\mathcal K)}$ and ${\bf
    H}_{A(\mathcal K)}$ (Streicher's and Krivine's triposes
  respectively). \end{theo}
\begin{proof} We use Corollary \ref{coro:equivalenciaconlaiota2}, taking the map $\iota$ as the double perpendicular closure operator. The
verification of the hypothesis is almost identical to the case of
Corollary \ref{coro:equivalenciaconlaiota}. \end{proof} We also omit
the proof of the next corollary as it is very similar to the proof of
Corollary \ref{coro:fyginversas}.

\begin{coro}\label{coro:almost} Let $f:{\mathcal
    A}\to{\mathcal B}$ be an invertible arrow in ${\bf IOCA}$ with
  inverse $g$, then it induces an equivalence between ${\bf
    H}_{\mathcal A}$ and ${\bf H}_{\mathcal B}$ as indexed
  HPO. \end{coro}

\item \label{item:krivineandstreicher} We end this section by showing
  that in a manner similar than before, one can complete the proof of
  the equivalences mentioned in (b.) at the beginning  of Section
  \ref{section:IAtoTripos} and prove a version of Theorem
  \ref{theo:equivalenciaenadjuntos} for the triposes associated to
  $\mathcal B$ and $A_\perp(K(\mathcal B))$. This result
  appeared in \cite[Theorems 6.5, 6.6]{kn:ocar2}. We present a
  slightly more detailed proof in order to illustrate that in this
  case the equivalence of triposes comes from a bijective set
  theoretical map. Recall that the underlying set of
  $A_\perp(K(\mathcal B))$ is $\mathcal P_\perp(B)=\{C \subseteq B:
  \iota_0(C)=({}^{\perp}C)^\perp=C\}$ if $B$ is the underlying set of
  $\mathcal B$ and where in this case the perpendicularity relation is
  simply the inequality relation $\leq$ of $B$ (see Paragraph
  \ref{basicdiagram},\ref{item:AK}). In \cite[Lemma 6.10]{kn:ocar2} it
  is proved that $\iota_0(C)=\uparrow(\inf C)$ so that $\mathcal
  P_\perp(B)$ consists in the set of principal filters of $B$ with
  respect to its order.

 \begin{theo}\label{theo:equivalenciaenadjuntos2} Let $\mathcal
    B$ be a IOCA, then the triposes associated with $\mathcal B$ and
    $A_{\perp}(K(\mathcal B))$ are
    equivalent. \end{theo} \begin{proof} We prove this result by
   considering the following maps (that are almost the same than the
   ones in \ref{theo:equivalenciaenadjuntos}) $f:\mathcal
   P_{\perp}(B)\to B$ as $f(P)=\bigcurlywedge P$ and $g:B\to \mathcal
   P_{\perp}(B)$ as $g(a)= {\uparrow} a$ in the adapted Corollary
   \ref{coro:almost}. In this case, the situation is
   simpler than before because $f$ and $g$ are inverses of each
   other. We observed therein that $fg=\operatorname{id}_B$ and for $P
   \in \mathcal P_\perp(B)$ we have that $gf(P)={\uparrow}\inf
   P=\overline{P}=P$ and then $gf=\operatorname{id}$. The proof of the
   rest of the conditions follows closely the corresponding proof for
   the case or implicative algebras in Theorem
   \ref{theo:equivalenciaenadjuntos}. \end{proof}

\end{list}  

\begin{thebibliography}{99} 
\bibitem{kn:birkhoff} Birkhoff, G. {\em Lattice theory} Colloquium Publications, Vol 25, American
Mathematical Society, (1995), Third edition. 

\bibitem{kn:bor1} Borceux, F. \emph{Handbook of Categorical Algebra Volume 1. Basic Category
Theory}, Encyclopedia of Mathematics and its Applications, 50. Cambridge Univ. Press. Cambridge.(2008) 

\bibitem{kn:bor2} Borceux, F. \emph{Handbook of Categorical Algebra Volume 2. Categories and Structures}, Encyclopedia of Mathematics and its Applications, 51. Cambridge
Univ. Press. Cambridge.(2008) 

\bibitem{kn:report} Ferrer Santos, W., Guillermo, M. and Malherbe, O. {\em A Report on realizability},
arXiv:1309.0706v2 [math.LO], 2013, pp. 1--25. 

\bibitem{kn:ocar}Ferrer Santos, W., Frey, J.,Guillermo, M. and Malherbe, O., Miquel, A. {\em
Ordered combinatory algebras and realizability}. Mathematical Structures in Computer Science, Camb. Univ. Press (2015) pp.: 1--31

\bibitem{kn:ocar2}Ferrer Santos, W., Guillermo, M. and Malherbe,
  O. {\em Realizability in ordered combinatory algebras with
    adjunction}. Mathematical Structures in Computer Science,
    Camb. Univ. Press (2019) vol 29, 3, pp. 430--464, DOI
    10.1017/S0960129518000075.

\bibitem{kn:gralattice}Gr\"atzer, G. \emph{General Lattice Theory},
  Mathematische Reihe Lehrbücher und Monographien aus dem Gebiete der
  Exakten Wissenschaften 52, Birkhäuser Basel (1978).

\bibitem{kn:hofstra2003} Hofstra, P.J. and van Oosten,
J., {\em Ordered partial combinatory algebras.} Math. Proc. Cam. Phil. Soc. 134 (2003) 445--463. 

\bibitem{kn:hofstra2006} Hofstra, P.\, {\em All realizability isrelative}, Math. Proc. Cambridge Philos. Soc. 141 (2006), no. 2, pp. 239--264. 

%
\bibitem{kn:hofstra2008} Hofstra, P. {\em Iterated realizabilityas a comma construction}, Math. Proc. Cambridge Philos. Soc. 144(2008), no.
1, pp. 39--51. 


\bibitem{kn:hyland} Hyland, J.M.E. {\em The effective topos}, Proc. of The L.E.J. Brouwer Centenary Symposium
(Noordwijkerhout 1981) pp. 165-216, North Holland 1982. 

\bibitem{kn:tripos} Hyland, J.M.E., Johnstone, P.T., Pitts, A.M. {\em Tripos theory},
Math. Proc. Cambridge Phil. Soc. 88 (1980), pp 205--232. 

\bibitem{kn:kr2001} Krivine, J.-L. {\em Types lambda-calculus
  inclassical Zermelo-Fraenkel set theory}, Arch. Math. Log. 40
  (2001), no. 3, pp. 189--205.

\bibitem{kn:kr2003} Krivine, J.-L. {\em Dependent choice, ‘quote’ and
  the clock}, Th. Comp. Sc. 308 (2003), pp. 259--276.

\bibitem{kn:kr2008} Krivine, J.-L. {\em Structures de
  r\'ealisabilit\'e, RAM et ultrafiltre sur $\mathbb
  N$},\,(2008). http : //www.pps.jussieu.fr/~krivine/Ultrafiltre.pdf.

\bibitem{kn:kr2009} Krivine, J.-L. {\em Realizability in classical
  logic in Interactive models of computation and program behaviour},
  Panoramas et synth\`eses 27 (2009), SMF.
\bibitem{kn:cwm} Mac Lane, S. {\em Categories for the Working Mathematician}. Graduate Texts in Mathematics. Vol 5 (1971). Springer. 
\bibitem{kn:implimiquel}Miquel, A. {\em Implicative algebras for non
  commutative forcing}, in Workshop. Mathematical Structures in
  Computation, Lyon,
  2014.//smc2014.univ-lyon1.fr/lib/exe/fetch.php?media=miquel.pdf

\bibitem{kn:implimiquel2} Miquel, A. {\em Implicative algebras: a new
  foundation for forcing and realizability}, in
  Workshop. Realizability in Uruguay, Piriapolis, 2016.
  //www.p\'edrot.fr/montevideo2016/miquel-slides.pdf

\bibitem{kn:newmiquel} Miquel, A. \emph{Implicative algebras: a new
  foundation for forcing and realizability},
  https://arxiv.org/abs/1802.00528, (February 2018).

\bibitem{kn:emthesis} Miquey, E. {\em R\'ealisabilit\'e
classique et effets de bords}, http://www.theses.fr/s180391 (2017). 

\bibitem{kn:street} Street, R. {\em The formal theory of monads},
  Journal of Pure and Applied Algebra 2. 1972 (149-168). North-Holland
  Publishing Company.

\bibitem{kn:streicher} Streicher, T. {\em Krivine'\/s
  Classical Realizability from a Categorical Perspective},
  Math. Struct. in Comp. Science, vol 23, No. 6, 2013.

\bibitem{kn:vanOosbook} van Oosten, J. {\em Realizability, an
  Introduction to its Categorical Side}, (2008), Elsevier.

\bibitem{kn:OostenZou2016} van Oosten, J., Zou T. {\em Classical and
  Relative Realizability}, Theory and applications of categories, Vol 31, 2016, No. 22, (571-593). 

\end{thebibliography}
\end{document}